\documentclass[final]{siamart}
\usepackage{amsmath,amssymb,amsfonts,stmaryrd}
\usepackage{soul}
\usepackage{array}
\usepackage{empheq}
\usepackage{xfrac}
\usepackage[belowskip=-5pt]{subcaption}
\usepackage{minibox}
\usepackage{enumitem}
	\setlist{nosep} \usepackage{color}
\usepackage{hyperref}
\usepackage{cite}
\usepackage{booktabs}
\usepackage{multirow}

\newsiamthm{claim}{Claim}
\newsiamremark{remark}{Remark}
\newsiamremark{expl}{Example}
\newsiamremark{conjecture}{Conjecture}
\crefname{hypothesis}{Hypothesis}{Hypotheses}

\makeatletter
\let\oldnl\nl\newcommand{\nonl}{\renewcommand{\nl}{\let\nl\oldnl}}\makeatother

\newcommand{\h}{\hat}

\allowdisplaybreaks
\newcommand{\tcb}{\textcolor{black}}

\newcommand{\TheTitle}{Nonsymmetric Reduction-based\\Algebraic Multigrid}
\newcommand{\TheAuthors}{Manteuffel, M\"unzenmaier, Ruge, and Southworth }
\headers{Nonsymmetric Reduction-based AMG}{\TheAuthors}
\title{{\TheTitle}\thanks{This research was conducted with Government support under and awarded by DoD, Air Force Office of Scientific Research, National Defense Science and Engineering Graduate (NDSEG) Fellowship, 32 CFR 168a. This work was performed under
the auspices of the U.S. Department of Energy under grant numbers (SC) DE-FC02-03ER25574 and (NNSA) DE-NA0002376, and
  Lawrence Livermore National Laboratory under contract B614452.
  }}

\author{  Thomas~A.~Manteuffel
  \thanks{Department of Applied Mathematics,
          University of Colorado at Boulder.}
  \and
  Steffen M\"unzenmaier
  \thanks{Fakult\"at f\"ur Mathematik, Universit\"at Duisburg-Essen.}
  \and
  John Ruge
  \footnotemark[2]
  \and
  Ben~S.~Southworth
  \thanks{Department of Applied Mathematics,
          University of Colorado at Boulder
          (\email{ben.s.southworth@gmail.com}).}
}

\ifpdf\hypersetup{  pdftitle={\TheTitle},
  pdfauthor={\TheAuthors}
}
\fi

\begin{document}
\maketitle

\begin{abstract}
Algebraic multigrid (AMG) is often an effective solver for symmetric positive definite (SPD) linear systems resulting
from the discretization of general elliptic PDEs, or the spatial discretization of parabolic PDEs.
However, convergence theory and most variations of AMG rely on $A$ being SPD. Hyperbolic PDEs, which arise often in large-scale
scientific simulations, remain a challenge for AMG, as well as other fast linear solvers, in part because the resulting
linear systems are often highly nonsymmetric. Here, a novel convergence framework is developed for nonsymmetric,
reduction-based AMG, and sufficient conditions derived for $\ell^2$-convergence of error and residual. In particular,
classical multigrid approximation properties are connected with reduction-based measures to develop a robust
framework for nonsymmetric, reduction-based AMG. 

Matrices with block-triangular structure are then recognized as being amenable to reduction-type algorithms, and
a reduction-based AMG method is developed for upwind discretizations of hyperbolic PDEs, based on the concept
of a Neumann approximation to ideal restriction ($n$AIR). $n$AIR can be seen as a variation of local AIR ($\ell$AIR)
introduced in previous work, specifically targeting matrices with triangular structure. Although less versatile than $\ell$AIR,
setup times for $n$AIR can be substantially faster for problems with high connectivity. $n$AIR is shown to be an effective
and scalable solver of steady state transport for discontinuous, upwind discretizations, with unstructured meshes,
and up to 6th-order finite elements, offering a significant improvement over existing AMG methods. 
$n$AIR is also shown to be effective on several classes of ``nearly triangular'' matrices, resulting from
curvilinear finite elements and artificial diffusion.
\end{abstract}

\section{Introduction} \label{sec:intro}

Solving large, sparse linear systems is fundamental to the numerical solution of partial differential equations (PDEs). For
high-dimensional PDEs, even a moderate resolution of the discrete problem can lead to enormous problem sizes, which
require highly efficient, parallel solvers. Ultimately, it is important that a solver is
both \textit{algorithmically scalable} (\textit{fast}), with a cost in floating point operations linear or near-linear with the number of
unknowns, and \textit{scalable in parallel} (\textit{scalable}), where these operations can be efficiently distributed across many processors.
For symmetric positive definite (SPD) matrices, such as those that arise in the discretization of elliptic PDEs or the spatial
discretization of a parabolic PDE in time, a number of fast, scalable, iterative or direct methods have been developed. However,
there remains a lack of fast, scalable solvers for the highly nonsymmetric matrices that arise in the discretization of
hyperbolic PDEs and full space-time discretizations of general PDEs. A subset of such highly nonsymmetric matrices are
block-triangular matrices, with blocks that are small enough to invert directly, which is one focus of this work. 

For PDEs of elliptic type, algebraic multigrid (AMG) is among the fastest class of linear solvers. When applicable, AMG converges
in linear complexity with the number of degrees-of-freedom (DOFs), and scales in parallel like $O(\log_2(P))$, up to hundreds of
thousands of processors, $P$ \cite{Baker:2012ko}. Originally, AMG was designed for essentially SPD linear systems, and convergence of AMG
is relatively well understood in the SPD setting \cite{Brandt:1985um,Falgout:2004cs,Vassilevski:2008wd,ideal18,
brannick2018optimal}. Nonsymmetric matrices
pose unique difficulties for AMG in theory and in practice. In particular, coarse-grid correction, a fundamental part of AMG's fast
convergence, is generally not a contraction in the nonsymmetric setting, meaning that coarse-grid correction can
increase error. There have been efforts to develop AMG theory and methods in the nonsymmetric setting \cite{Brezina:2010dm,
Manteuffel:2017,Notay:2000vy,Sala:2008cv,Wiesner:2014cy,Notay:2000vy,Lottes:2017jz,Mense:2008gj,Notay:2010em,notay2018}.
However, the theoretical understanding of nonsymmetric AMG remains limited, with very few convergence bounds proved in norm,
and there has yet to be a robust and scalable AMG solver for highly nonsymmetric problems.

For the most part, previous work on nonsymmetric AMG has appealed to traditional AMG thought, where coarse-grid
correction captures (right) singular vectors associated with small singular values and relaxation attenuates error
associated with large singular values. Here, we take a different, reduction-based approach, appealing to the premise
that certain ``ideal'' restriction and interpolation
operators can lead to a reduction-based (nonsymmetric) AMG method. Although a reduction-based solver is not a new
concept, here we recognize (i) an important class of linear systems for which reduction can be highly effective, and (ii) 
develop a theoretical framework explaining the convergence of nonsymmetric reduction-based AMG in norm.

Background on sparse triangular systems, reduction, and reduction-based AMG is given in Section \ref{sec:background}.
A novel convergence framework is then developed for nonsymmetric, reduction-based AMG
in Section \ref{sec:analysis}, including sufficient conditions for $\ell^2$-convergence of error and residual,
and a formal connection of reduction-based AMG to classical multigrid approximation properties.
Section \ref{sec:tri} then
recognizes block-triangular and near-triangular matrices to be well-suited for reduction-based AMG, and discusses
such operators in the context of discontinuous discretizations of advection.
In particular, for hyperbolic-type PDEs, a Neumann approximation to the ideal restriction operator ($n$AIR) provides
an accurate, sparse approximation, and an AMG algorithm referred to as $n$AIR is developed based on this principle. 
This is a similar
approach to that developed in \cite{air2} ($\ell$AIR), but $n$AIR can offer setup-times that are orders of magnitude faster than $\ell$AIR
in some cases. Steady state transport is used as a model problem, which arises in large-scale simulation of neutron and radiation
transport \cite{Adams:2013to, Bailey:2009tk,Colomer:2013iv,Hawkins:2012wb, Park:2013ju}. Results in Section \ref{sec:results}
show that $n$AIR outperforms current state-of-the art methods, and is able to attain an order-of-magnitude reduction in residual in only
1--2 iterations, even for high-order discretizations on unstructured grids.

\section{Reduction-based algebraic multigrid}\label{sec:background}

\subsection{Triangular matrices and parallelism}\label{sec:background:tri}

Sparse matrices with triangular structure arise in a number of interesting settings. In contrast to elliptic and parabolic PDEs,
the solution of hyperbolic PDEs lies on characteristic curves, and the solution at any point depends only on the solution upwind
along its given characteristic. This allows for very steep gradients or ``fronts'' to propagate through the domain.
Due to such behavior, discontinuous, upwind discretizations are a natural approach to discretizing many hyperbolic 
PDEs \cite{Brezzi:2004hf,Morel:2005tv,Morel:2007jj,Ragusa:2012gn}. For a fully upwinded discretization, the resulting matrix
has a block-triangular structure in some (although potentially unknown) ordering. Implicit time-stepping schemes or steady state
solutions then require the solution of such linear systems. There has also been growing interest in parallel-in-time solvers.
Although most work on these has been geometric in nature, one can also consider algebraically solving the sparse
matrix equations associated with a full space-time discretization. Such discretizations using an explicit
time-stepping scheme (for any PDE) result in triangular matrices, and an implicit time-stepping scheme coupled with an
upwind discretization of a hyperbolic PDE also leads to a block triangular matrix. 

Solving linear systems with block-triangular structure is easy in serial using a forward or backward solve.
However, there are cases where a block-triangular solve arises in the parallel setting. In particular, any time
an upwind advective discretization is either (i) part of a larger PDE discretization that cannot be stored on a small number of
processors, or (ii) coupled to other variables that are not inverted easily and require parallel preconditioners, a triangular-type
solve is necessary in parallel. Scheduling algorithms have been developed for sparse matrices that 
can add some level of parallelism to this process, but such algorithms are primarily relevant for shared memory
environments \cite{Alvarado:1993bl,Anderson:1989,Li:2013dm,Liu:2016ew}, and, from the perspective of simulation
of PDEs, efficiency in a distributed-memory environment is fundamental. In the simplest case of a perfectly structured processor
grid of squares/cubes over some $d$-dimensional domain, and a fixed, constant direction of flow, a forward solve in parallel scales like $O(P^{1/d})$ for
$P$ processors \cite{Bailey:2009tk}. However, for non-constant flow or non-uniform processor grids with respect to the flow and domain, this
convergence can suffer. In fact, for \textit{any} processor configuration over a given domain, it is straightforward 
to construct a velocity field for advection such that a distributed forward solve takes $P$ communication steps to
complete. Even if each step of this process is fast, linear or even square-root scaling in $P$ is poor parallel performance.

Iterative-type methods offer an alternative (potentially) more amenable to parallelism. 
However, Krylov and most other traditional iterative-type methods are
generally either divergent or too slow (with respect to convergence) to be considered a tractable approach, and 
solving large, sparse, block-triangular linear systems in (distributed-memory) parallel environments remains a difficult
task. Geometric multigrid has been applied to hyperbolic PDEs and upwind discretizations using the well-known line-smoother
approach \cite{Oosterlee:1998ih, Amarala:2013bx, Yavneh:1998fw}, but such an approach requires a fixed/known
direction over which to relax and has limits in terms of parallel scalability. When considering time-dependent hyperbolic
PDEs, explicit time-stepping schemes can be used to avoid solving linear systems. But, explicit schemes suffer from stability
constraints, such as the Courant-Friedrichs-Lewy (CFL) condition, which often require extremely small time steps, a process
that is sequential and can limit performance in the parallel setting. AMG is known to scale like $O(\log P)$ in parallel,
and an AMG solver that is effective on block triangular and near-triangular matrices would overcome many of the aforementioned
difficulties.

\subsection{Reduction-based AMG}\label{sec:background:amg}

Reduction consists of solving of a problem by equivalently solving multiple smaller problems. In thinking about triangular
systems, a direct forward or backward solve is a reduction algorithm:  starting with a system of size $n\times n$, eliminate
one DOF, and reduce the problem to size $(n-1)\times(n-1)$. Although this is a sequential algorithm, it suggests that
reduction is an effective and tractable approach for triangular systems, which we demonstrate in this paper. 

Given an $n\times n$ matrix $A$, suppose the $n$ DOFs are partitioned into $n_c$ C-points and $n_f$ F-points. For explanation, 
assume that the unknowns have been ordered with the F-points followed by the C-points. In this context, a well-known
example of reduction is a $2\times 2$ block LDU decomposition:
\begin{align}\label{eq:red3}
\begin{bmatrix} A_{ff} & A_{fc} \\ A_{cf} & A_{cc}\end{bmatrix}^{-1} & = 
	\begin{bmatrix} I & -A_{ff}^{-1}A_{fc} \\ 0 & I\end{bmatrix}
	\begin{bmatrix} A_{ff}^{-1} & 0 \\ 0 & \mathcal{K}_A^{-1}\end{bmatrix}
	\begin{bmatrix} I & 0 \\ -A_{cf}A_{ff}^{-1} & I\end{bmatrix}.
\end{align}
Here, solving $A\mathbf{x}=\mathbf{b}$ is reduced to solving two systems, $A_{ff}\in\mathbb{R}^{n_F\times n_F}$ and
the Schur complement, $\mathcal{K}_A := A_{cc}-A_{cf}A_{ff}^{-1}A_{fc}\in\mathbb{R}^{n_C\times n_C}$. LDU
decompositions assume that the action of $A_{ff}^{-1}$ is available to compute the action of the first and third
matrix blocks in \eqref{eq:red3}, while in practice it is typically not. However, approximating such a decomposition
is the goal behind numerous preconditioners. In fact, a two-level reduction-based AMG method can be posed as a
variant of a block LDU decomposition, which is what we develop here by algebraically approximating the operators
that lead to reduction in \eqref{eq:red3}: $A_{cf}A_{ff}^{-1}$ and $A_{ff}^{-1}A_{fc}$.

Algebraic multigrid applied to the linear system $A\mathbf{x} = \mathbf{b}$, $A\in\mathbb{R}^{n\times n}$, consists of
two processes: relaxation and coarse-grid correction, designed to be complementary in the sense that they effectively
reduce error associated with different  parts of the spectrum of $A$. Relaxation often takes the form
\begin{align}
\mathbf{x}_{k+1} = \mathbf{x}_k + M^{-1}(\mathbf{b} - A\mathbf{x}_k),\label{eq:relax}
\end{align}
where $M^{-1}$ is some approximation to $A^{-1}$ such that the action of $M^{-1}$ can be easily computed.
Coarse-grid correction typically takes the form
\begin{align}
\mathbf{x}_{k+1} = \mathbf{x}_k + P(RAP)^{-1}R(\mathbf{b} - A\mathbf{x}_k), \label{eq:cgc}
\end{align}
where $P\in\mathbb{R}^{n\times n_c}$ and $R\in\mathbb{R}^{n_c\times n}$ are interpolation and restriction operators,
respectively, between $\mathbb{R}^n$ and the next coarser grid in the AMG hierarchy, $\mathbb{R}^{n_c}$, and back. 
Denote the projection
$\Pi := P(RAP)^{-1}RA$. A two-level $V(1,1)$-cycle is then given by combining coarse-grid
correction in \eqref{eq:cgc} with pre- and post-relaxation steps as in \eqref{eq:relax}, resulting in a two-grid
error propagation operator of the form
\begin{align*}
E_{TG} = (I - M_{\textnormal{post}}^{-1}A)(I - \Pi)(I - M_{\textnormal{pre}}^{-1}A).
\end{align*}

A classical AMG approach is used here, where a CF-splitting of DOFs defines the coarse grid \cite{Brandt:1985um,ruge:1987}.
Operators $A,P$, and $R$ can then be written in block form:
\begin{align}
A = \begin{bmatrix} A_{ff} & A_{fc} \\ A_{cf} & A_{cc}\end{bmatrix}, \hspace{3ex}
P = \begin{bmatrix} W \\ I\end{bmatrix}, \hspace{3ex}
R = \begin{bmatrix} Z & I\end{bmatrix}, \label{eq:block}
\end{align}
where $W\in\mathbb{R}^{n_f\times n_c}$ interpolates to F-points via linear combinations
of coarse-grid DOFs, and $Z\in\mathbb{R}^{n_c\times n_f}$ restricts F-point residuals. Note
that \eqref{eq:block} implicitly assumes the same CF-splitting for $R$ and $P$, although
the sparsity patterns for nonzero elements of $Z^T$ and $W$ may be different. For notation, denote the
coarse-grid operator:\footnote{$\mathcal{K}$ is used to denote the coarse-grid operator instead of the traditional notation,
$A_c,$ to avoid confusion with subscripts denoting C-points.}
\begin{align}
\mathcal{K} :&= RAP = ZA_{ff}W + A_{cf}W + ZA_{fc}+A_{cc}. \label{eq:rap}
\end{align}

Define the ``ideal restriction'' and ``ideal interpolation'' operators as
\begin{align*}
R_{\textnormal{ideal}} := \begin{bmatrix} -A_{cf}A_{ff}^{-1} & I \end{bmatrix} ,\hspace{3ex} 
P_{ideal} := \begin{bmatrix} -A_{ff}^{-1}A_{fc}  \\ I \end{bmatrix},
\end{align*}
These operators define the LDU reduction in \eqref{eq:red3}. Note that $\mathcal{K} := \mathcal{K}_A$ with $P=P_{\textnormal{ideal}}$
\textit{or} $R = R_{\textnormal{ideal}}$, independent of the other. Ideal interpolation has been explored in the context
of reduction-based geometric multigrid methods \cite{Foerster:1981,Ries:1983,Falgout:2014el,bno:16},
and is well-motivated under classical AMG theory for SPD matrices, where it is optimal in a certain sense with respect
to two-grid convergence \cite{Falgout:2004cs,ideal18}. Ideal restriction was also discussed in the context of reduction
in \cite{Falgout:2004cs} and, in \cite[Section 2.3]{air2}, shown to be the unique restriction operator that gives an exact
coarse-grid correction at C-points, independent of interpolation. This result, along with a corollary on ideal interpolation
\cite[Section 2.2]{air2}, are summarized in the following results. 

\begin{lemma}[Ideal restriction]\label{th:ideal}
For a given CF-splitting, assume that $A_{ff}$ is nonsingular and let $A,P$, and $R$ take the block form as given in
\eqref{eq:block}. Then, an exact coarse-grid correction at C-points is attained for all $\mathbf{e}$ if
and only if $R = R_{\textnormal{ideal}}$. Furthermore, the error in coarse-grid correction is given by
\begin{align}
(I - \Pi)\mathbf{e} & = \begin{bmatrix} \mathbf{e}_f - W\mathbf{e}_c \\ \mathbf{0}\end{bmatrix}.\label{eq:idealerr}
\end{align}
A coarse-grid correction using $R_{\textnormal{ideal}}$ followed by an exact solve on F-points, results in an exact
two-grid solver, independent of $W$. 
\end{lemma}

\begin{corollary}[Ideal interpolation]\label{cor:ideal}
For a given CF-splitting, assume that $A_{ff}$ is nonsingular and let $A,P$, and $R$ take the block form as given in
\eqref{eq:block}. Then, an exact solve on F-points, followed by a coarse-grid correction using $P := P_{ideal}$,
yields an exact two-level solver, independent of $Z$.
\end{corollary}

Thus, in the nonsymmetric setting, ideal transfer operators are ``ideal'' in a reduction sense: when coupled with an exact solve on
F-points, $R_{\textnormal{ideal}}$ and $P_{ideal}$ each lead to an exact two-level method, independent of the accompanying interpolation
and restriction operators, respectively. Note that the ordering of solving the coarse- and fine-grid problems is important: 
\tcb{in order to be an exact reduction,}
the F-point solve must \textit{follow} coarse-grid correction with $R_{\textnormal{ideal}}$, while the F-point solve must \textit{precede}
coarse-grid correction with $P_{ideal}$.

\subsection{Relation to existing methods}\label{sec:background:other}

Schur-complement preconditioning and reduction-based solvers are not new to the literature. Numerous algorithms
have been based on a block LDU decomposition \eqref{eq:red3} and approximate Schur complement (for example,
\cite{ Carvalho:2001ib, Mandel:1990kv,Mense:2008gj,Mense:2008vx,
Notay:2000vy,Saad:1999km,benzi2005numerical}). Reduction has also been studied in the multigrid and AMG
context, originally in the geometric setting \cite{Foerster:1981,Ries:1983}, more recently algebraically
\cite{Brannick:2010uq,MacLachlan:2006gt}, and also
as the basis for the multigrid reduction-in-time (MGRIT) parallel-in-time method \cite{Falgout:2014el,Dobrev:2016vc,bno:16}.
MGRIT is designed for nonsymmetric problems, but is geometric in nature and relies on the very specific matrix structure
that arises in time integration, more or less a block 1D advection problem. The AMG developments in
\cite{Brannick:2010uq,MacLachlan:2006gt} are fully algebraic and reduction-based, but assume a Galerkin coarse grid,
meaning restriction is defined as $R = P^T$. For the highly nonsymmetric systems considered here, this is typically not
a good choice \cite{nonsymm,air2}, motivating a Petrov-Galerkin method, where $R\neq P^T$. Unfortunately, choosing
$R \neq P^T$ also introduces new difficulties, because if $R$ and $P$ are not ``compatible'' in some sense, the norm of
coarse-grid correction can be $\gg 1$ \cite{nonsymm,air2}, leading to a divergent method. 

Approximating the ideal restriction or interpolation operators has also been considered in other (non-reduction-based)
AMG methods, including \cite{Manteuffel:2017,Olson:2011fg,Wiesner:2014cy}. Also motivated in a block-LDU sense,
ideal restriction and interpolation are approximated in \cite{Wiesner:2014cy} by performing a constrained minimization
over a fixed sparsity pattern for $R$ and $P$. A similar constrained minimization approach
for nonsymmetric systems was used in \cite{Manteuffel:2017,Olson:2011fg}, where ideal operators of
$A^*A$ and $AA^*$ for $P$ and $R$, respectively, are approximated using a constrained minimization.
In these cases, the solvers appeal to more classical convergence
theory by enforcing constraints to interpolate certain (known) vectors associated with small singular values exactly.

The AIR algorithm, developed here and in \cite{air2}, takes a somewhat converse approach. Building on the discussion
in Section \ref{sec:background:amg}, AIR relies on an accurate approximation to $R_{ideal}$, and couples this
with an accurate F-relaxation scheme to achieve a reduction-based algorithm. In particular, the focus is on reduction
through the ideal restriction operator. In \cite[Lemma 1]{air2}, it is shown that for $R\neq R_{ideal}$, care must
also be taken when building $P$ to ensure a stable coarse-grid correction. Theory developed here indicates that when
$R$ approximates $R_{\textnormal{ideal}}$, building interpolation to accurately capture right singular vectors associated
with small singular values is sufficient for a stable coarse-grid correction, as well as two-grid convergence in the $\ell^2$-
and $A^*A$-norms. How accurately the derived conditions require $P$ to interpolate modes depends on how accurately
$R\approx R_{ideal}$. 

This paper can be seen as a companion paper to the $\ell$AIR method developed in \cite{air2}. The theory
developed here is more general than that in \cite{air2}, and provides rigorous explanation as to why better
interpolation methods are needed when considering advection-diffusion-reaction \cite{air2} compared with pure
advection-reaction (here and \cite{air2}). Conversely, the method developed here is less general than $\ell$AIR,
effective primarily on the purely advective or nearly advective case, but as a result, also has a significantly cheaper setup
cost. The basic idea is that $\ell$AIR approximates ideal restriction by solving many small, dense, linear
systems, which can be moderately expensive as the problem size grows. $n$AIR recognizes that, for advective-type
discretizations, a similar approximation can be obtained by a few sparse matrix products. A detailed algorithmic
comparison of $n$AIR and $\ell$AIR can be found in \cite{air2}.

\tcb{Finally, theory on convergence of nonsymmetric AMG in the $\sqrt{A^*A}$-norm was developed simultaneously
with this work in \cite{nonsymm}. There, approximation properties are assumed on $R$ and $P$. Here, we replace
one approximation property assumption on either $R$ or $P$ with a measure of distance of $R$ or $P$ from ideal
restriction or interpolation, respectively, and consider convergence in the $\ell^2$- and $A^*A$-norms. If this
distance cannot be made small, then one should revert back to a framework as in \cite{nonsymm}, focusing on
more traditional approximation properties for $P$ and $R$. }

\section{Convergence of reduction-based AMG}\label{sec:analysis}

\subsection{Framework}\label{sec:analysis:framework}

Let $A$, $P$, and $R$ take the form introduced in \eqref{eq:block}.
Error propagation of coarse-grid correction is given by $I - \Pi$, where
$\Pi = P(RAP)^{-1}RA$ is a (generally non-orthogonal in any known inner product) projection onto the range of $P$ (Section \ref{sec:background}).
The motivation for AIR as an AMG algorithm is straightforward. Recall that
ideal restriction gives an exact approximation at C-points, independent of interpolation. Following this with a direct solve on
F-points gives an exact two-level method (Lemma \ref{th:ideal}). Although we do not expect ideal restriction
in practice, here we assume that an accurate approximation to $R_{\textnormal{ideal}}$ leads to an accurate solution
at C-points, which we follow with F-relaxation to distribute this accuracy to F-points.

Let $\Delta_F$ be an approximation to $A_{ff}^{-1}$. First, measures of the accuracy of F-relaxation as well as the
difference between ideal interpolation and restriction and the interpolation and restriction used in practice are defined,
respectively, as
\begin{align*} 
\delta_F &= I-\Delta_FA_{ff},\\ 
\delta_P &= A_{ff}W+A_{fc},\\
\delta_R &= ZA_{ff}+A_{cf}.
\end{align*}
\tcb{
Here, $\delta_R$ and $\delta_P$ are measured relative to $\|A\|$. Throughout the paper it will be assumed that
$A$ has been scaled so that $\|A\|\sim O(1)$. However, $\|A\|$ is explicitly carried through the derivation of bounds 
and proofs of convergence for completeness.
}

The error-propagation matrix associated with F-point relaxation is given by
\begin{align*} 
\mathcal{E}_F = \left[\begin{array}{cc} I & 0 \\ 0 & I \end{array}\right]
	- \left[\begin{array}{cc} \Delta_F & 0 \\ 0 & 0\end{array}\right]
	\left[ \begin{array}{cc} A_{ff} & A_{fc} \\ A_{cf} & A_{cc} \end{array}\right]
= \begin{bmatrix} \delta_F & -\Delta_FA_{fc} \\ \mathbf{0} & I\end{bmatrix}.
\end{align*}
The product of these two error matrices, $\mathcal{E} := \mathcal{E}_F(I-\Pi)$, can be put into a very convenient form \cite{air2},
\begin{align*}
\mathcal{E}
&= \left[\begin{array}{cc} I & 0 \\ 0 & I \end{array}\right]
	-\bigg( \left[\begin{array}{cc} \Delta_F & 0 \\ 0 & 0\end{array}\right]
	+ \left[ \begin{array}{cc} I - \Delta_F A_{ff} & - \Delta_F A_{fc} \\0& I \end{array}\right]
 	\left[ \begin{array}{c} W \\ I \end{array}\right] \\ & \hspace{6ex}
	\mathcal{K}^{-1}  \left[\begin{array}{cc}  Z & I  \end{array}\right] \bigg)
	\left[ \begin{array}{cc} A_{ff} & A_{fc} \\ A_{cf} & A_{cc} \end{array}\right]\\ 
&= \left[\begin{array}{cc} I & 0 \\ 0 & I \end{array}\right]
	- \left[\begin{array}{cc} I & \widehat{W} \\ 0 & I \end{array}\right]
	\left[\begin{array}{cc} \Delta_F & 0 \\ 0 & \mathcal{K}^{-1} \end{array}\right]
	\left[\begin{array}{cc} I & 0 \\ Z & I \end{array}\right]
	\left[ \begin{array}{cc} A_{ff} & A_{fc} \\ A_{cf} & A_{cc} \end{array}\right],
\end{align*}
where
\begin{align} \label{eqn:whW}
\widehat{W} = (I-\Delta_FA_{ff}) W -\Delta_FA_{fc} = \delta_F W - \Delta_FA_{fc}.
\end{align}
If $\Delta_F = A_{ff}^{-1}$, then $\widehat{W}$ becomes ideal interpolation.
The better $\Delta_F$ approximates $A_{ff}^{-1}$, the closer $\widehat{W}$ is to ideal
interpolation.  Here, $\widehat{W}$ is referred to as the ``effective interpolation'' of this method \cite{air2}.

Next, note that $ \mathcal{E}$ has the form
\begin{align}\label{eqn:MinvN}
 \mathcal{E} = I - M^{-1}A = M^{-1}(M-A),
\end{align}
where
\begin{align}
M &= \left[\begin{array}{cc} I & 0 \\ -Z & I \end{array}\right]\left[\begin{array}{cc} \Delta_F^{-1} & 0 \\ 0 & \mathcal{K} \end{array}\right]
\left[\begin{array}{cc} I & -\widehat{W} \\ 0 & I \end{array}\right]\nonumber\\
&= \left[\begin{array}{cc} \Delta_F^{-1} &  A_{fc} -(\Delta_F^{-1}-A_{ff})W\\-Z\Delta_F^{-1}   & \mathcal{K}+ Z\Delta_F^{-1}\widehat{W}\end{array}\right].\label{eq:M0}
\end{align}
A little extra work using \eqref{eq:rap} yields
\begin{align*}
\mathcal{K} + Z\Delta_F^{-1}\widehat{W} &= A_{cc} + A_{cf}W + ZA_{fc} + ZA_{ff}W +Z\Delta_F^{-1}(I-\Delta_FA_{ff}) W -Z A_{fc} \\
&= A_{cc} + A_{cf}W + ZA_{ff}W +Z(\Delta_F^{-1}- A_{ff}) W\\
&= A_{cc} + A_{cf}W + Z\Delta_F^{-1}W.
\end{align*}
Using \eqref{eq:M0},
\begin{align} 
M-A & =  \left[\begin{array}{cc} \Delta_F^{-1} - A_{ff} &  - (\Delta_F^{-1}-A_{ff})W\\
	-(Z\Delta_F^{-1}+A_{cf}) &   (Z\Delta_F^{-1}+A_{cf}) W \end{array}\right] \label{eq:M_A}\\
 & = \left[\begin{array}{c} \Delta_F^{-1} - A_{ff} \\ 
	-(Z\Delta_F^{-1}+A_{cf})\end{array}\right] 
	\left[\begin{array}{cc} I &  - W \end{array}\right].\label{eqn:M-A1}
\end{align}
Similarly, $M^{-1}$ can be expanded as
\begin{align}\nonumber
M^{-1} &= \left[\begin{array}{cc} I & \widehat{W} \\ 0 & I \end{array}\right]
	\left[\begin{array}{cc} \Delta_F & 0 \\ 0 & \mathcal{K}^{-1} \end{array}\right]
	\left[\begin{array}{cc} I & 0 \\ Z & I \end{array}\right] \\ \label{eqn:Minv1}
& = \left[\begin{array}{cc} \Delta_F + \widehat{W}\mathcal{K}^{-1}Z& \widehat{W}\mathcal{K}^{-1} \\ \mathcal{K}^{-1}Z & \mathcal{K}^{-1} \end{array}\right].
\end{align}

As mentioned in \eqref{eqn:MinvN}, the error-propagation matrix is given by $\mathcal{E} = M^{-1}(M-A)$.
The residual-propagation matrix is similar, given by $\mathcal{R} = A\mathcal{E}A^{-1} = (M-A)M^{-1}$. Each
of these can now be assembled to a convenient outer-product form.

For error propagation, combining (\ref{eqn:M-A1}) with (\ref{eqn:Minv1}) gives
\begin{align}
 \mathcal{E} &= \left[\begin{array}{cc} \Delta_F+\widehat{W}\mathcal{K}^{-1}Z & \widehat{W}\mathcal{K}^{-1} \\ \mathcal{K}^{-1}Z & \mathcal{K}^{-1} \end{array}\right]
	\left[\begin{array}{c} \Delta_F^{-1} - A_{ff} \nonumber\\
	-(Z\Delta_F^{-1}+ A_{cf})  \end{array}\right] 
	\left[\begin{array}{cc} I &  - W \end{array}\right]\\
&= \left[\begin{array}{c} (I-\Delta_F A_{ff}) - \widehat{W}\mathcal{K}^{-1}(ZA_{ff}+A_{cf}) \nonumber\\
-\mathcal{K}^{-1} (ZA_{ff}+A_{cf}) \end{array}\right]
 \left[\begin{array}{cc} I &  - W \end{array}\right] \\ 
 &= \left[\begin{array}{c} \delta_F - \widehat{W}\mathcal{K}^{-1}\delta_R \\
-\mathcal{K}^{-1} \delta_R \end{array}\right]
 \left[\begin{array}{cc} I &  - W \end{array}\right] .\label{eqn:E_E2}
\end{align}
Likewise, using (\ref{eqn:whW}), (\ref{eqn:M-A1}), and (\ref{eqn:Minv1}), the residual-propagation matrix is given by
\begin{align} 
\mathcal{R} &= \left[\begin{array}{c} \Delta_F^{-1} - A_{ff} \\ -(Z\Delta_F^{-1}+A_{cf})  \end{array}\right] 
	\left[\begin{array}{cc} \Delta_F(I- \delta_P\mathcal{K}^{-1}Z) &-\Delta_F\delta_P\mathcal{K}^{-1} \end{array}\right]. \label{eqn:E_R2}
\end{align}
The outer-product formulation provides a natural representation of powers of $\mathcal{E}^k$ and
$\mathcal{R}^k$, where
\begin{align}
\mathcal{E}^k &=  \left[\begin{array}{c} \delta_F - \widehat{W}\mathcal{K}^{-1}\delta_R \\ -\mathcal{K}^{-1} \delta_R \end{array}\right]
	G^{k-1} \left[\begin{array}{cc} I &  - W \end{array}\right] , \label{eqn:ETk} \\
\mathcal{R}^k & = \left[\begin{array}{c} \Delta_F^{-1} - A_{ff} \\ -(Z\Delta_F^{-1}+A_{cf})  \end{array}\right] G^{k-1}
	\left[\begin{array}{cc} \Delta_F(I- \delta_P\mathcal{K}^{-1}Z) &-\Delta_F\delta_P\mathcal{K}^{-1} \end{array}\right]. \label{eqn:RTk}
\end{align}
In particular, it is easily verified from \eqref{eqn:whW}, \eqref{eqn:E_E2} and \eqref{eqn:E_R2} that 
\begin{align} \nonumber
G &= \left[\begin{array}{cc} I &  - W \end{array}\right] 
 	\left[\begin{array}{c} \delta_F - \widehat{W}\mathcal{K}^{-1}\delta_R \\ \nonumber
	-\mathcal{K}^{-1} \delta_R \end{array}\right] \\ 
&= \delta_F + \Delta_F\delta_P\mathcal{K}^{-1} \delta_R \label{eqn:G1} \\
&= \delta_F + (I-\delta_F)A_{ff}^{-1}\delta_P\mathcal{K}^{-1} \delta_R  \label{eqn:G2}
\end{align}
is identical for $\mathcal{E}^k$ and $\mathcal{R}^k$. 

This is a fundamental observation for proving convergence-in-norm of reduction-based AMG. In particular, it is often
the case that $\|\mathcal{E}\|$ and/or $\|\mathcal{R}\|$ are greater than one. However, \eqref{eqn:ETk}
and \eqref{eqn:RTk} show that raising error- and residual-propagation to powers is equivalent to considering
powers of a different matrix, $G$ \eqref{eqn:G2}. Thus, bounding $\|G\|$ can lead to a convergent method, 
which is summarized in the following lemma. 

\begin{lemma}\label{lem:conv}
Let $W$, $Z$, and $\Delta_F$ be chosen such that 
\begin{align*}
\|G\| = \| \delta_F + \Delta_F \delta_P\mathcal{K}^{-1}\delta_R \| =\rho <1.0.
\end{align*}
Then, the iteration will converge with bounds
\begin{align*}
\| \mathbf{e}_k \| & \leq \rho^{k-1} \left\|\left[\begin{array}{c} \delta_F - \widehat{W}\mathcal{K}^{-1}\delta_R \\
	-\mathcal{K}^{-1} \delta_R \end{array}\right]\right\| \left\| \left[\begin{array}{cc} I &  - W \end{array}\right] \right\|\|\mathbf{e}_0\|,  \\
\| \mathbf{r}_k \| & \leq \rho^{k-1} \left\|\left[\begin{array}{c} \Delta_F^{-1} - A_{ff} \\ -(Z\Delta_F^{-1}+A_{cf})  \end{array}\right]\right\|
	\left\| \left[\begin{array}{cc} \Delta_F(I- \delta_P\mathcal{K}^{-1}Z) &-\Delta_F\delta_P\mathcal{K}^{-1} \end{array}\right]\right\|\|\mathbf{r}_0\|.
\end{align*}
\end{lemma}
\begin{proof}
The proof follows from \eqref{eqn:ETk}, \eqref{eqn:RTk}, \eqref{eqn:G1}, and the discussion above. 
\end{proof}
\tcb{
Note that the bound on $\|r_k\|$ is independent of $\|A\|$, but the separate terms are not. This can easily be adjusted
by scaling the second and third terms by $1/\|A\|$ and $\|A\|$, respectively. 
 }

In Section \ref{sec:tri:air}, we show that it is often possible in practice to construct $Z$ such that $\|\delta_R\|$ is quite small \tcb{relative to $\|A\|$}.
Recall that, \tcb{for hyperbolic problems,}  $\|\mathcal{K}^{-1}\| = O(1/h)$. On a relatively coarse grid, it is possible that $\|\delta_R\| \ll O(h)$ and, consequently,
$\| G \|<1.0$, regardless of $W$. In fact, $W=\mathbf{0}$ is a reasonable choice in that context because that choice reduces the
complexity of $\mathcal{K}$, making the algorithm more efficient. However, in general, a better $W$ may be
necessary for convergence. The following section develops conditions for which $\|G\| < 1$. 

\begin{remark}[Pre-F-relaxation]
Because ideal restriction gives an exact coarse-grid correction at C-points, thus far we have considered post-F-relaxation
to distribute this accuracy to F-points. If instead, $P$ is chosen to approximate $P_{\textnormal{ideal}}$, a pre-F-relaxation
scheme may be more appropriate (see Corollary \ref{cor:ideal}). It is easy to show that pre-F-relaxation enjoys the same asymptotic
behavior as post-F-relaxation.
\end{remark}

\subsection{Two-grid convergence}\label{sec:analysis:2grid}

In this section, conditions are derived to bound $\|G\| \leq \rho < 1$. The focus of this work is on problems for which
$\|\delta_F\|$, and $\|\delta_R\|$ or $\|\delta_P\|$, can be made small \tcb{relative to $\|A\|$}. However, for a given family of discretizations, 
$\|\delta_R\|$ and $\|\delta_P\|$ are typically fixed, independent of $h$; that is, the accuracy of approximation to
ideal operators does not improve as $h\to0$. To bound $\|G\|<1$ as $h\to 0$, additional measures
must be taken to account for the term $\delta_P\mathcal{K}^{-1}\delta_R$ in \eqref{eqn:G2}, because
$\|\mathcal{K}^{-1}\| \sim O(1/h)$. To do so, we consider a classical multigrid ``weak approximation property''
for $P$ and $R$.

\begin{definition}[WAP on $P$ with respect to SPD $\mathcal{A}$]
An interpolation operator, $P$, satisfies the \textit{weak approximation property} (WAP) with respect to SPD matrix $\mathcal{A}$, with
constant $C_P$ if, for any $\mathbf{v}$ on the fine grid, there exists a $\mathbf{w}_c$ on the coarse grid such that
\label{assp:waponP}
\begin{align*}
\|\mathbf{v} - P\mathbf{w}_c\|^2 & \leq \frac{C_P}{\|\mathcal{A}\|} \langle \mathcal{A}\mathbf{v},\mathbf{v}\rangle. 
\end{align*}
\end{definition} 

Recall that the Schur complement of $A$ is $\mathcal{K}_A = A_{cc} - A_{cf}A_{ff}^{-1}A_{fc}$. 
Comparing the coarse-grid operator \eqref{eq:rap} with the Schur complement yields
\begin{align}   \nonumber
\mathcal{K} - \mathcal{K}_A &= (ZA_{ff}+ A_{cf}) A_{ff}^{-1}A_{fc} + (ZA_{ff}+ A_{cf})W \\ \nonumber
&= (ZA_{ff}+ A_{cf})A_{ff}^{-1}(A_{ff}W +A_{fc}) \\ \label{eqn:K-K_A}
&= \delta_RA_{ff}^{-1}\delta_P. 
\end{align}
Now, assume that $P$ satisfies the WAP with respect to $\mathcal{A} = A^*A$, with constant $C_P$. Then, for every
$\mathbf{v} = (\mathbf{v}_f^T, \mathbf{v}_c^T)^T$,
\begin{align}\label{eqn:WAP1}
\inf_{\mathbf{w}_c} \left\| \left(\begin{array}{c} \mathbf{v}_f\\\mathbf{v}_c \end{array}\right) - 
	\left(\begin{array}{c} W \\ I \end{array}\right)\mathbf{w}_c \right\|^2
\leq \frac{C_P}{\|A^*A\|} \| A \mathbf{v} \|^2.
\end{align}  
Let $\h{\mathbf{w}}_c$ satisfy the infimum above. Then,
\begin{align*}
\| \mathbf{v}_f - W \mathbf{v}_c\| & \leq \|\mathbf{v}_f - W\h{\mathbf{w}}_c \| + \| W(\h{\mathbf{w}}_c - \mathbf{v}_c) \|\\
	& \leq \|\mathbf{v}_f - W\h{\mathbf{w}}_c \| + \| W\|\|\h{\mathbf{w}}_c - \mathbf{v}_c \|.
\end{align*}
Noting from \eqref{eqn:WAP1} that $\|\mathbf{v}_f - W\h{\mathbf{w}}_c \|^2 +
\|\mathbf{v}_c - \h{\mathbf{w}}_c\|^2 \leq \frac{C_P}{\|A^*A\|}\| A \mathbf{v} \|^2$, we can form a constrained maximization problem and bound
\begin{align*}
\| \mathbf{v}_f - W \mathbf{v}_c\| & \leq \sqrt{C_P(1+\|W\|^2)}\frac{\|A\mathbf{v}\|}{\|A\|}
	 := C_W \frac{\|A\mathbf{v}\|}{\|A\|},
\end{align*}
where $C_W := \sqrt{C_P(1+\|W\|^2)}$. In particular, let $\mathbf{v}_f = -A_{ff}^{-1}A_{fc} \mathbf{v}_c$. Then,
\begin{align}\label{eqn:WAPKA}
\|A_{ff}^{-1}\delta_P\mathbf{v}_c\| = \| (W + A_{ff}^{-1} A_{fc})\mathbf{v}_c\| \leq C_W\frac{ \| A \mathbf{v}\|}{\|A\|}
= C_W\frac{\| \mathcal{K}_A \mathbf{v}_c\|}{\|A\|}.
\end{align}
Following from \eqref{eqn:K-K_A} and \eqref{eqn:WAPKA}, observe that
\begin{align*}
\| \mathcal{K} \mathbf{v}_c \| & \geq \|\mathcal{K}_A \mathbf{v}_c\| - \| (\mathcal{K}_A - \mathcal{K}) \mathbf{v}_c\| \\
& = \|\mathcal{K}_A\mathbf{v}_c\| - \|\delta_RA_{ff}^{-1}\delta_P\mathbf{v}_c\| \\
& \geq (1-C_W\frac{\|\delta_R\|}{\|A\|}) \| \mathcal{K}_A \mathbf{v}_c\|.
\end{align*}
Note that, given $W$ with WAP constant $C_W$, we may choose $Z$ so that $C_W\|\delta_R\| < \|A\|$.
Later, $C_W\|\delta_R\|$ will be chosen slightly smaller. Combining, we arrive at
\begin{align*}
\|A_{ff}^{-1}\delta_P \mathcal{K}^{-1}\| =  \sup_{\mathbf{v}_c\neq\mathbf{0}} \frac{\| A_{ff}^{-1}\delta_P \mathbf{v}_c\| }{\|\mathcal{K} \mathbf{v}_c\|} \leq\frac{C_W}{\|A\|-C_W\|\delta_R\|}.
\end{align*}

This result is generalized in the following Lemma. 

\begin{lemma}\label{lem:waps}
Suppose $P$ satisfies the WAP with respect to $A^*A$, with constant $C_P$. Then, if $Z$ is chosen such that $C_W\|\delta_R\| < \|A\|$,
\begin{align*}
\|A_{ff}^{-1}\delta_P \mathcal{K}^{-1}\| & \leq \frac{C_W}{\|A\|-C_W\|\delta_R\|}, \hspace{5ex}
\|\delta_P \mathcal{K}^{-1}\| \leq \frac{C_W\|A_{ff}\|}{\|A\|-C_W\|\delta_R\|},
\end{align*}
where $C_W = C_P\sqrt{1+\|W\|^2}$. 
{\color{black}
Similarly, suppose $R^*$ satisfies the WAP with respect to $AA^*$, with
constant $C_R$. Define $C_Z = \sqrt{C_R(1+\|Z\|^2)}$ and assume that $C_Z\|\delta_P\| < \|A\|$. 
Then,
\begin{align*}
\|\mathcal{K}^{-1}\delta_RA_{ff}^{-1}\| & \leq \frac{C_Z}{\|A\|-C_Z\|\delta_P\|}, \hspace{5ex}
\|\mathcal{K}^{-1}\delta_R\| \leq \frac{C_Z\|A_{ff}\|}{\|A\|-C_Z\|\delta_P\|}.
\end{align*}}
\end{lemma}
\begin{proof}
The results follow from the above discussion and noting that $\|\delta_P \mathcal{K}^{-1}\| \leq \|A_{ff}\|\|A_{ff}^{-1}\delta_P \mathcal{K}^{-1}\|$.
Results for $R$ follow a similar derivation as that for $P$.
\end{proof}

The above discussion is summarized in the following results.

\begin{theorem} \label{thm:WAPP} 
Let $A$, $P$, and $R$ be defined as above. If $W$ is chosen so that $P$ satisfies the WAP with respect to $A^*A$, with
constant $C_P$, then $Z$ can be chosen to approximate $-A_{cf}A_{ff}^{-1}$ and $\Delta_F$ can be chosen
to approximate $A_{ff}^{-1}$ so that
\begin{align*}
\|G\| & \leq \| \delta_F \| + \Big(1+\|\delta_F\|\Big) \frac{C_W\|\delta_R\|}{\|A\|-C_W\|\delta_R\|} = \rho < 1, 
\end{align*}
where $C_W = C_P\sqrt{1+\|W\|^2}$.
\end{theorem}
\begin{proof}
The proof follows from Lemma \ref{lem:waps} and the bound using \eqref{eqn:G1},
\begin{align*}
\|G\| & \leq \|\delta_F\| + (1+\|\delta_F\|)\|\delta_R\|\|A_{ff}^{-1}\delta_P\mathcal{K}^{-1}\|.
\end{align*}
\end{proof}

{\color{black}
\begin{theorem} \label{thm:WAPR} 
Let $A$, $P$, and $R$ be defined as above. If $Z$ is chosen so that $R$ satisfies the WAP with respect to $AA^*$
and constant $C_R$, then $W$ can be chosen to approximate $-A_{ff}^{-1}A_{fc}$ and $\Delta_F$ can be chosen
to approximate $A_{ff}^{-1}$ such that
\begin{align*}
\|G\| & \leq \| \delta_F \| + (1+\|\delta_F\|) \frac{C_Z\|\delta_P\|}{\|A\|-C_Z\|\delta_P\|} = \rho < 1, 
\end{align*}
where $C_Z = C_R\sqrt{1+\|Z\|^2}$.
\end{theorem}
\begin{proof}
The proof follows from the discussion above and the proof of Theorem \ref{thm:WAPP}.
\end{proof}
}

Theorems \ref{thm:WAPP} \tcb{and \ref{thm:WAPR}}  give insight into the roles of restriction, interpolation, and F-relaxation.
F-relaxation can help convergence bounds, but only to a certain extent. For an exact solve on F-points, $\|\delta_F\| = 0$. Then, 
for example, in Theorem \ref{thm:WAPP}, to ensure that $\|G\| \leq \rho < 1$, we must have
$C_W\|\delta_R\|/\|A\| \leq \frac{\rho}{\rho+1} < \frac{1}{2}$. This can be accomplished both through a more accurate interpolation
with respect to the WAP or a more accurate approximation to ideal restriction. As $\|\delta_F\|$ increases, that is, F-relaxation
becomes less effective, interpolation and restriction must improve through reduced $C_W$ and/or $\|\delta_R\|$.

The $\ell^2$-convergence of error and residual also follows from Theorems \ref{thm:WAPP} and \ref{thm:WAPR} and
Lemma \ref{lem:conv}. Although it is possible that $\|\mathcal{E}\|,\|\mathcal{R}\| > 1$, if the hypothesis are satisfied,
there exists $k_1$ and $k_2$ such
that $\|\mathcal{E}^{k_1}\|, \|\mathcal{R}^{k_2}\| < 1$. Iterations before these values of $k$ are reached may appear to be
diverging, but they will eventually converge with asymptotic factor $\rho$. How long it takes to reach asymptotic convergence
depends on the other matrix blocks in $\mathcal{E},\mathcal{R}$. This has been observed in practice.

Consider $\| \mathcal{E}\|$ and $\|\mathcal{R}\|$ with respect to the size of the mesh. From (\ref{eqn:E_E2}), it is clear
that choosing $\mathbf{e} = (\mathbf{0},\mathbf{e}_c^T)^T$ yields $\|\mathcal{E} \| \geq \| \mathcal{K}^{-1}\delta_R\|$. In the
case that $A$ is a discrete approximation of a PDE, $\|\mathcal{K}^{-1}\|$ may grow with $n$, the size of the system, whereas,
$\|\delta_R\|$ may be fixed. Although $\|G\|=\rho$ is independent of $n$, without additional approximation properties
on $Z$, the norm of $\mathcal{E}$ may not be bounded independent of $n$. This would lead to a method for which it
takes more iterations to reach a given accuracy as the problem size increases. Building on Lemma
\ref{lem:waps}, conditions for residual and error propagation bounded independent of $n$ are given in the following
corollary.

\begin{corollary}[Bounded Residual and Error]\label{lem:boundedRandE}
Assume that $P$ satisfies the WAP with respect to $A^*A$, with constant $C_P$, Theorem \ref{thm:WAPP} holds, and
the condition numbers of $A_{ff}$ and $\Delta_F$ are independent of problem size. Then, for $k\geq 0$,
$\|\mathcal{R}^k\|$ is bounded, independent of problem size, and converges
with asymptotic rate $\leq \rho$. 

\tcb{
If, in addition, $R^*$ satisfies a WAP with respect to $AA^*$, then for $k\geq 0$, 
$\|\mathcal{E}^k\|$ is bounded, independent of problem size and converges with 
asymptotic rate $\leq\rho$.}
\end{corollary}
\begin{proof} 
\tcb{Consider the terms in \eqref{eqn:RTk}. Under the assumption that $P$ satisfies a WAP with respect to 
$A^*A$, it was shown in Lemma \ref{lem:waps} that $\|\delta_P\mathcal{K}^{-1}\|$ is bounded independent of $n$. 
All other terms in the equation are bounded independent of $n$.
Likewise, consider the terms in \eqref{eqn:ETk}. From Lemma \ref{lem:waps}, if $R^*$ satisfies a WAP with respect to $AA^*$,
then $\|\mathcal{K}^{-1}\delta_R\|$ is bounded independent of $n$.}
\end{proof}

\begin{remark}[C-point relaxation]\label{rem:cpt}
In a two-level setting, adding relaxation over C-points as part of the pre- or post-relaxation scheme offers little to no
improvement of convergence. Suppose $Z = -A_{cf}\Delta_R$, where $\Delta_F = \Delta_R$, that is, the same
approximation is used for F-relaxation as for approximating $R_{ideal}$. Now, consider following the F-point relaxation
with a C-point relaxation. Similar to F-point relaxation, the error-propagation operator associated with
C-point relaxation is given by
\begin{align*}
\mathcal{E}_C = \left[\begin{array}{cc} I & 0 \\ 0 & I \end{array}\right] -
	\left[\begin{array}{cc} 0 & 0 \\ 0 & \Delta_C\end{array}\right]
	\left[ \begin{array}{cc} A_{ff} & A_{fc} \\ A_{cf} & A_{cc} \end{array}\right],
\end{align*}
where $\Delta_C$ is an approximation to $A_{cc}^{-1}$. Noting that the C-point rows of $(M-A)$ are zero when
\tcb{$Z = -A_{cf}\Delta_F$ (see \eqref{eq:M_A})}, multiplying by $\mathcal{E}$ yields
\begin{align*}
\mathcal{E}_C \mathcal{E} &= I - \left(M^{-1} + \left[\begin{array}{cc} 0 & 0 \\ 0 & \Delta_C\end{array}\right](M-A)M^{-1})\right) A \\
&= I - M^{-1} A = \mathcal{E}_E.
\end{align*}
This demonstrates that, in the context of a two-grid method, \tcb{with reasonable choice of $Z$}, C-point relaxation does not improve the 
solution. In the multilevel setting, C-point relaxation can offer some improvement in convergence, but remains much less important than 
F-relaxation.
\end{remark}

\section{The triangular case}\label{sec:tri}

Building on the previous section, this section develops an accurate approximation to $R_{\textnormal{ideal}}$
for matrices with block-triangular or near-triangular structure. Numerical examples
demonstrate the accuracy of nAIR in the context of convergence constants derived in Section \ref{sec:analysis}.
Although reduction can be
achieved with ideal restriction or interpolation, focusing on ideal restriction allows for coupling the $n$AIR method
with established interpolation methods.
One result here is that error and residual propagation of $n$AIR are nilpotent in the case of block-triangular matrices,
which compensates for inaccurate interpolation near boundaries.

{\color{black}
As motivation, consider a block-discontinuous discretization of a steady state advection or advection-reaction 
equation in two dimensions,
\begin{align*}
\mathbf{b}\cdot \nabla u + c(x,y) u & = f,
\end{align*}
for arbitrary velocity field $\mathbf{b}(x,y)$ (without cycles, or else the problem is not well posed), forcing
function $\mathbf{f}$, reaction field $c(x,y)$, and some inflow and outflow boundary conditions (for example,
see Section \ref{sec:results}).
Suppose a uniform square mesh is used in two dimensions. Then, for many discretizations, such as
discontinuous Galerkin, among others, DOFs in each element of the mesh
depend on exactly two other elements in the mesh, specifically the two elements upwind with respect to the
direction of flow, and each element in the mesh corresponds to a non-overlapping block in the matrix. 
In the multigrid context, consider a block red-black coarsening scheme, where C-points and
F-points represent entire finite elements, as shown in Figure \ref{fig:rb}.
\begin{figure}[h!]
\centering
\includegraphics[width=2.5in]{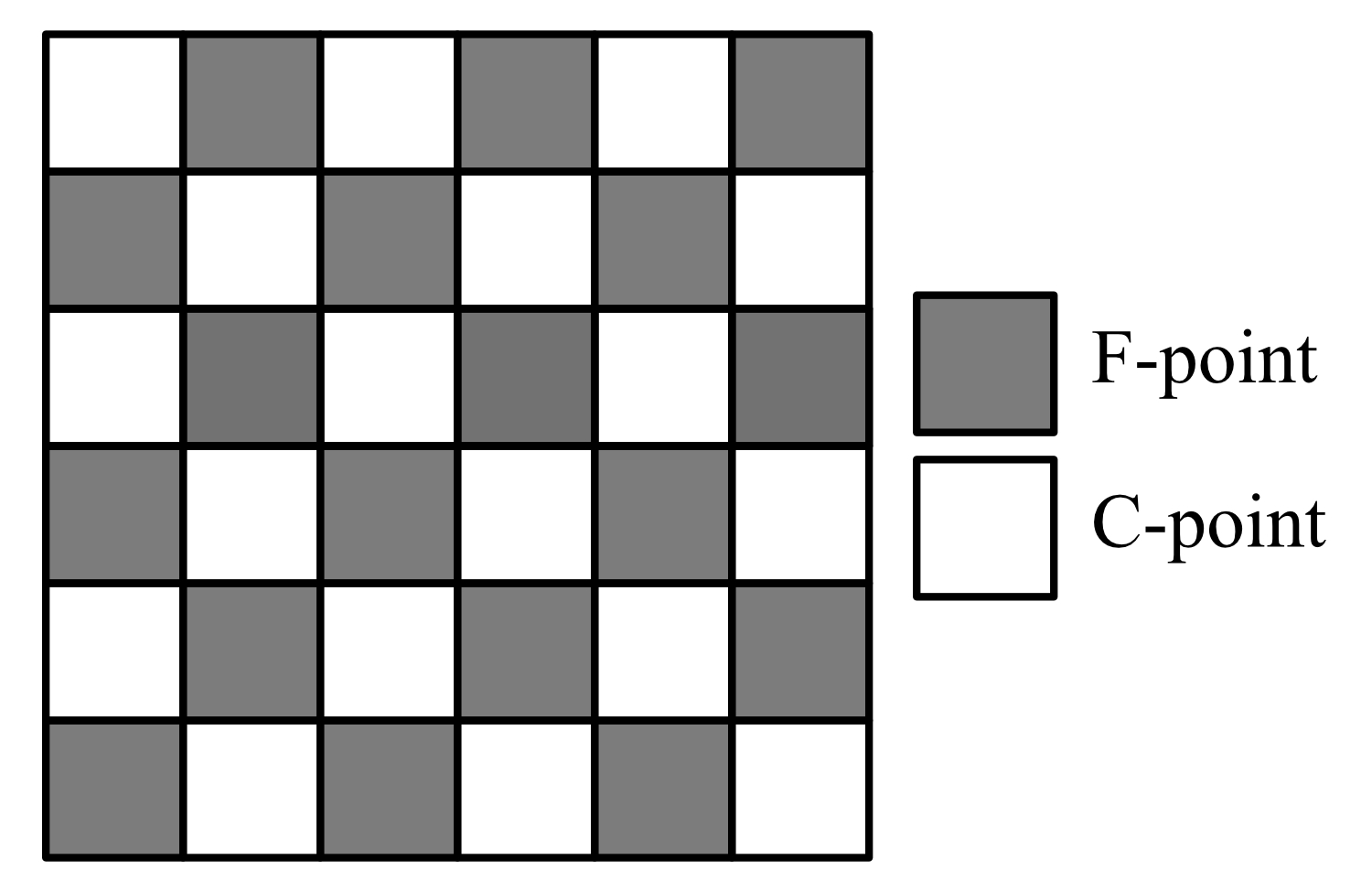}
\caption{Element-wise ``block'' red-black (or white-gray) coarsening of uniform structured grid in two-dimensions.}
\label{fig:rb}
\end{figure}

For linear finite elements, the DOFs corresponding to a block discretization are often located at the vertices of the
mesh, with one DOF contained in each element touching a given vertex. For higher-order discretizations, there are
more DOFs, but the underlying principle remains the same: all DOFs in a given finite element only depend on
DOFs within that element, or DOFs in either of the two upwind elements. Looking at Figure \ref{fig:rb}, note that for
any velocity field $\mathbf{b}(x,y)$, all F-point blocks depend only on C-point blocks, and C-point blocks depend
only on F-point blocks. If $A$ is scaled to have unit diagonal, then it follows that the submatrix $A_{ff} = I$, 
and thus, ideal restriction and exact F-relaxation
are trivial to compute in practice.

Of course, this example holds for a block diffusion discretization as well (that is, $A_{ff} = I$ there as well). However,
for advection, the coarse-grid operator $RAP = A_{cc} - A_{cf}A_{ff}^{-1}A_{fc} = I - A_{cf}A_{fc}$ maintains a 
similar structure to the fine grid. Note that for any C-point, the corresponding row of $RAP=I - A_{cf}A_{fc}$ is nonzero in
C-point blocks that can be reached through a C-F-C path in the graph. Generally for an advection problem, (i) each C-point
has about three coarse-grid connections, (ii) all connections are upwind (and, thus, strictly lower triangular in the
matrix), and (iii) at least one of these connections is essentially cross-stream, making it a weak connection. In this
case, coarsening similar to Figure \ref{fig:rb} based on strong connections leads to an $A_{ff}\neq I$, but for
which $A_{ff}^{-1}$ is well conditioned and easily approximated with a sparse matrix. In contrast, each coarse-grid point in a diffusion
discretization is connected to eight other points, making $A_{ff}^{-1}$ and $R_{ideal}$ difficult to approximate effectively
in a sparse manner, and multilevel reduction much more difficult.

On unstructured meshes, non-quadrilateral elements, higher dimensions, or coarser grids in an AMG hierarchy,
it is typically not the case that $A_{ff}=I$. Nevertheless, similar principles suggest that $A_{ff}^{-1}$
can be approximated efficiently, which is confirmed directly in Section \ref{sec:tri:const} and implicitly in numerical
results of Section \ref{sec:results}.
}

\subsection{Neumann approximate ideal restriction}\label{sec:tri:air}

For general matrices, a naive and often ineffective approach to approximate $A_{ff}^{-1}$ is to use a truncated
Neumann expansion. However, in the case of block-triangular matrices, particularly those resulting from the discretization
of differential operators, a truncated Neumann inverse expansion can provide a remarkably accurate approximation.
For ease of notation, assume that $A$ has been scaled to have unit diagonal,
and suppose we have determined a
CF-splitting (or block CF-splitting if $A$ is block lower triangular). Then, let $A_{ff} =
I - L_{ff}$, where $L_{ff}$ is the strictly lower triangular part of $A_{ff}$. Because $L_{ff}$
is nilpotent, $A_{ff}^{-1}$ can be written as a finite Neumann expansion:\footnote{Because there are no cycles in the graph
of $L_{ff}$, the degree of nilpotency is given by the maximum graph distance between any two F-points, say $d_f$. In the case
of an acyclic graph, this is equivalent to the \textit{longest path problem}. For discretizations of differential operators that result
in an acyclic graph, it is generally the case that $d_f \ll n_f$.}
\begin{align}\label{eq:neumann}
A_{ff}^{-1} = (I - L_{ff})^{-1} = \sum_{i=0}^{d_f+1} L_{ff}^i.
\end{align}
To approximate $A_{ff}^{-1}$, we consider an order-$k$ approximation given by truncating \eqref{eq:neumann}:
$\Delta^{(k)} := \sum_{i=0}^k L_{ff}^i$, for some $0\leq k\leq d_f$. Define a restriction operator
based on a Neumann approximate ideal restriction ($n$AIR):
\begin{align*}
R := \begin{bmatrix}-A_{cf}\Delta^{(k)} & I\end{bmatrix}.
\end{align*}
Error in $\Delta^{(k)}$ as an approximate inverse can be measured as $I - \Delta^{(k)}A_{ff} = L_{ff}^{k+1}$, which
gives a measure of how accurately we approximate $R_{\textnormal{ideal}}$.

Note that the error relation $I - \Delta^{(k)}A_{ff} = L_{ff}^{k+1}$ does not require $A_{ff}$ to be lower triangular. However,
triangular structure is fundamental to $L_{ff}^k$ being small as $k$ increases, particularly when considering the discretization of
differential operators. Consider $L_{ff}$ as the adjacency
matrix of a directed acyclic graph. Then $(L_{ff}^k)_{ij}$ gives the sum of weighted walks of length $k$ from vertex $i$ to vertex $j$ (weight
given by \textit{product} of the walk's edges). Thus, we are interested in the number of F-F connections and size of the weights. For the
discretization of differential operators, off-diagonal elements are typically small relative to the diagonal, and an AMG CF-splitting is chosen to eliminate strong
F-F connections. In the case of triangular matrices, such as an upwind discretization of advection, regardless of the problem dimension, there
only exist walks from node $i$ to nodes $j$ \textit{downstream} of $i$, in the direction along the characteristic. This means that the sparsity pattern of
$L_{ff}^k$ only reaches out to neighbors in effectively one direction. Thus, as $k$ increases, the number of neighbors within distance $k$
should not increase significantly, while the product of edges should decay rapidly. 

\begin{figure}[!b]
  \centering
  \begin{subfigure}[b]{0.495\textwidth}
    \includegraphics[width=\textwidth]{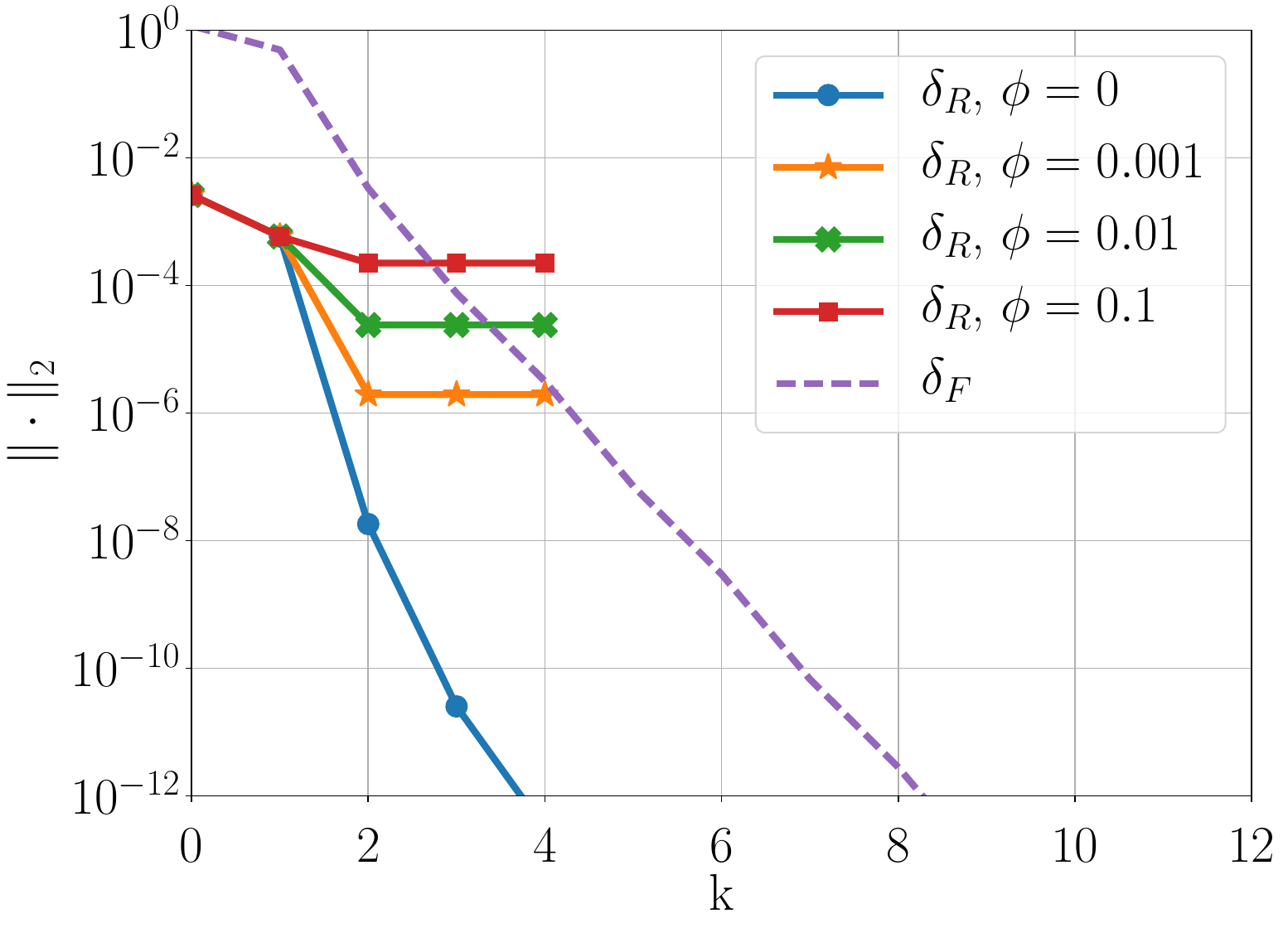}
    \caption{1st-order, $\kappa= 0$}
  \end{subfigure}
   \begin{subfigure}[b]{0.495\textwidth}
    \includegraphics[width=\textwidth]{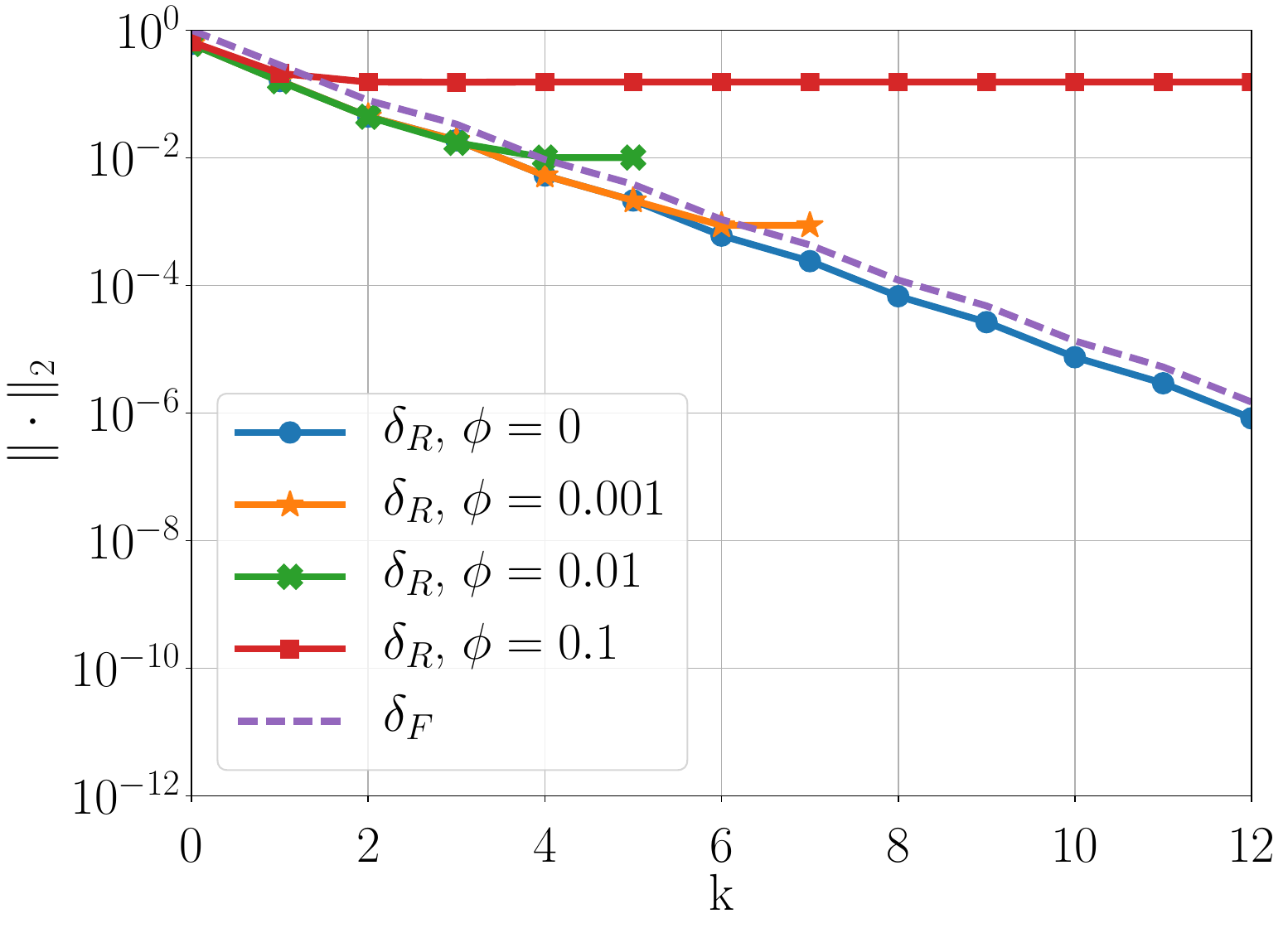}
    \caption{1st-order, $\kappa= 1$}
  \end{subfigure}
  \\\vspace{2ex}
  \begin{subfigure}[b]{0.495\textwidth}
    \includegraphics[width=\textwidth]{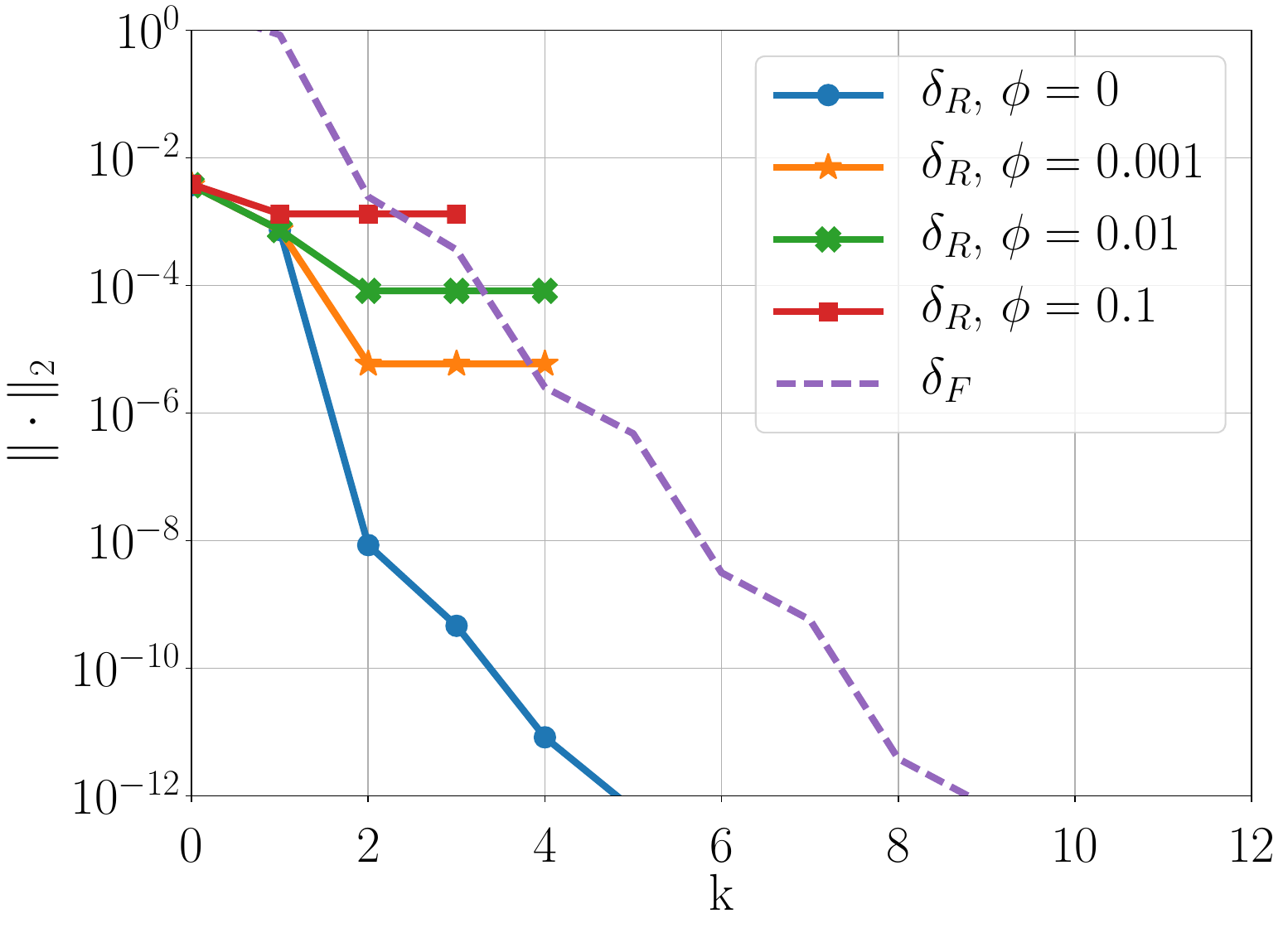}
    \caption{4th-order, $\kappa= 0$}
  \end{subfigure}
   \begin{subfigure}[b]{0.495\textwidth}
    \includegraphics[width=\textwidth]{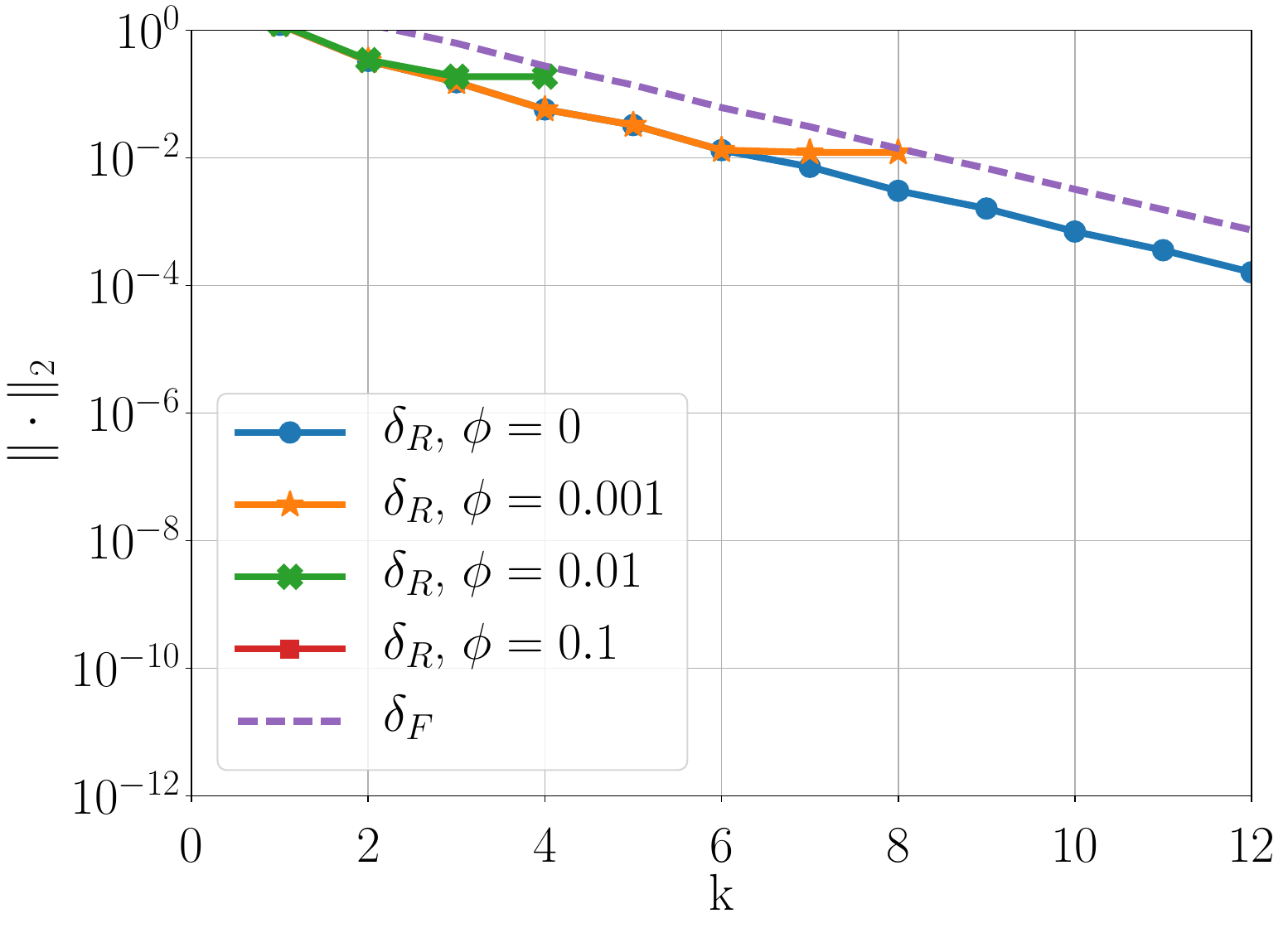}
    \caption{4th-order, $\kappa= 1$}
  \end{subfigure}  
      \caption{Using upwind DG, $\|\delta_R\|$ and $\|\delta_F\|$ are displayed as a function of $k$, corresponding to degree of 
      Neumann expansion for $n$AIR
      and number of (block) smoothing iterations for Jacobi F-relaxation. The top row corresponds to linear finite elements and the bottom
      to 4th-order finite elements, while the left column corresponds to steady state transport (no diffusion, $\kappa=0$) and the right columns
      corresponds to a diffusion-dominated advection-diffusion-reaction equation, $\kappa = 1$. $\delta_R$ is a function of parameter $\phi < 1$,
      where a Neumann expansion is only performed on strong connections (for row $i$, $a_{ij}$ such that $|a_{ij}| \geq \phi \sup_j |a_{ij}|$).     
       In all cases, the discretization has $\approx 400,000$ DOFs, about half of which are F-points.  }
  \label{fig:neumann}
\end{figure}

\begin{remark}[Nilpotent Error Propagation]
{\color{black}
It is straight forward to show that if $A$ is lower triangular in some ordering 
and $\Delta_R$, $\Delta_P$, and $\Delta_F$ are truncated Neumann approximations to $A_{ff}^{-1}$,
then two-grid error propagation based on Jacobi F-relaxation is strictly lower triangular.
Moreover,  multilevel error propagation, with coarse-grid correction based on $n$AIR and 
Jacobi F-relaxation, is also strictly lower triangular and, thus, nilpotent. 

Although, the degree of nilpotency is not sufficiently small to be considered practical, nilpotency of error 
propagation  is directly impactful near the boundaries of the domain,  where DOFs actually fall off the nilpotent cliff 
in $O(1)$ iterations. AMG often struggles with interpolation near domain boundaries, and the nilpotent behavior of
$n$AIR eliminates this problem. This benefit has been observed in practice.

}
\end{remark}

{\color{black}
\subsection{Evaluating constants}\label{sec:tri:const}

This section considers nAIR applied to discontinuous Galerkin discretizations of the steady state transport equation
and an advection-diffusion-reaction equation, with diffusion coefficient $\kappa$. Both discretizations are on unstructured
meshes with a moving (non-constant) velocity field and material discontinuities of $10^8$ between different subdomains.
Further details on the transport
discretization can be found in Section \ref{sec:results}. The advection-diffusion-reaction equation follows an analogous
derivation, and is discussed in \cite{air2}.
Recall, when approximating ideal restriction (as opposed to ideal interpolation), the theoretical constants of interest are
\begin{align*} 
\delta_F &= I-\Delta_FA_{ff},\hspace{5ex} 
\delta_R = ZA_{ff}+A_{cf},
\end{align*}
as well as the constant for which $P$ satisfies a WAP (see Theorem \ref{thm:WAPP}). The purpose of this section
is to demonstrate that, on an interesting and difficult model problem, $n$AIR leads to \tcb{$\|\delta_F\| \ll 1, \|\delta_R\| \ll \|A\|$,
where $A$ has been scaled so that $\|A\|\sim O(1)$.)}
Following from Theorem \ref{thm:WAPP}, we then show that strong two-grid convergence is provable under
minimal assumptions on $P$ for the transport equation. Results also show why the reduction approach is
successful for advection discretizations, but less effective for diffusive problems.

Figure \ref{fig:neumann} plots $\|\delta_F\|$ and $\|\delta_R\|$ as a function of number of iterations and degree of Neumann
expansion, respectively. In practice, it is usually not a good idea to form transfer operators using neighbors further than
distance two, unless coarsening aggressively, due to the rapid increase in the number of nonzeros in coarse-grid operators.
Similarly, transfer operators are typically only formed based on strong connections in the matrix, again to limit coarse-grid
fill in, particularly in the multilevel setting. Here, strong connections for each row are defined as entries larger
than $\phi$ times the largest row element (positive
or negative). 

Looking at the left column of Figure \ref{fig:neumann}, we see that for both 1st- and 4th-order discretizations of
the steady state transport problem, $n$AIR based on distance-one or -two strong neighbors is able to achieve
$\|\delta_R\|$ between $10^{-5}-10^{-3}$. Furthermore, three iterations of F-relaxation leads to $\|\delta_F\| \approx 10^{-4}$,
and four iterations improves this to $\|\delta_F\| \approx 10^{-6}$. In the diffusion-dominated case, $\kappa = 1$ (right column),
results are not as good. For linear elements, distance-two $n$AIR achieves at best $\|\delta_R\| \approx 0.1$ and four iterations
of F-relaxation only yields $\|\delta_F\| \approx 0.01$, while the 4th-order discretization is worse.  

Table \ref{tab:dr} plots two-grid convergence bounds from Theorem \ref{thm:WAPP} as a function of constants $C_W$ and
$\|\delta_R\|$, for a fixed $\|\delta_F\|$. Note that due to small $\|\delta_R\|$ and $\|\delta_F\|$, only mild assumptions must
be made on interpolation for rapid two-grid convergence. Although constants are likely to increase somewhat in the multilevel setting,
the initial discussion in Section \ref{sec:tri} suggests that good approximations to $A_{ff}^{-1}$ and $R_{\textnormal{ideal}}$
can still be obtained on coarser levels in the hierarchy. 

{
\begin{table}[h!]
\begin{center}
\renewcommand{\tabcolsep}{6pt}
\begin{tabular}{|c | c c c c c c c c |}\toprule
$\|\delta_R\|$ / $C_W$ & 1 & 5 & 10 & 25 & 50 & 100 & 250 & 500 \\\midrule
$10^{-5}$ & 0.0001 & 0.0002 & 0.0002 & 0.0004 & 0.0006 & 0.001 & 0.002 & 0.005 \\
$10^{-4}$ & 0.0002 & 0.0006 & 0.0011 & 0.0026 & 0.0051 & 0.010 & 0.026 & 0.053 \\
$10^{-3}$ & 0.0011 & 0.0051 & 0.0102 & 0.0257 & 0.0527 & 0.111 & 0.334 & -- \\
0.01 & 0.0102 & 0.0527 & 0.1112 & 0.3335 & -- & -- & -- & -- \\
0.1~ &  0.1112 & -- & -- & -- & -- & -- & -- & -- \\\bottomrule
\end{tabular}
\caption{Theoretical two-grid convergence bounds as a function of $\|\delta_R\|$ and $C_W$, assuming $\|\delta_F\| = 10^{-4}$
\tcb{and $A$ has been scaled such that $\|A\| \sim O(1)$}.}
\label{tab:dr} 
\end{center}
\end{table}
}
}

{\color{black}
The constant $C_P$, where $C_W := \sqrt{C_P(1+\|W\|^2)}$, is evaluated in \cite{nonsymm} for a hyperbolic
steady state transport equation, discretized with SUPG and DG finite element discretizations (for details, see
Section \ref{sec:results} here or \cite{nonsymm}). There, modified classical AMG interpolation \cite{DeSterck:2008fc}
is shown to have constant $C_P$ of 157 and 204 for SUPG and DG, respectively. Typically $\|W\|\gtrapprox 1$.
Say $\|W\| = 2$. Then $C_W\approx 28$ and $32$. Applying the same tests to the one-point interpolation used here
yields constants $C_P$ on the order of $300$ and $360$, and $C_W\approx 39$ and 42. In practice, we find that
the sparser structure of one-point interpolation is advantageous in the multilevel setting, and we see equivalent
convergence with the two methods, despite the larger constant $C_W$ using one-point interpolation.

Interestingly, in \cite{nonsymm} an analogous algorithm to AIR for interpolation is numerically shown to
have constants $C_P\approx 10-20$ and $C_W\approx 7$ and 10. Furthermore, tests suggest the constants may
be independent of problem size, a property which may not be the case for one-point and classical
interpolation. Note that those tests did
not use an efficient algorithm to build the interpolation, and extending AIR-like ideas to interpolation for hyperbolic-type
problems is ongoing work. However, such numbers indicate that better interpolation operators (than existing methods)
can be constructed for hyperbolic-type problems, and, with such interpolation, two-grid convergence could be
obtained with fairly weak approximations of $R_{\textnormal{ideal}}$.
}

\section{Algorithm}\label{sec:alg}

The main components of $n$AIR follow that of a standard Petrov-Galerkin multigrid scheme with no pre-relaxation.
Section \ref{sec:alg:amg} introduces details on the AMG components of $n$AIR, including parameters and routines for
strength-of-connection, CF-splitting, interpolation, restriction, and relaxation. Other details, including support of block
structure in matrices and a filtering procedure to reduce complexity, are discussed in Section \ref{sec:alg:extra}.

\subsection{AMG components}\label{sec:alg:amg}

Effective F-relaxation and an accurate approximation to ideal restriction both require $A_{ff}$ to be relatively well conditioned.
This is consistent with motivation for a classical AMG strength-of-connection and CF-splitting \cite{Brandt:1985um,ruge:1987},
which are used here. Jacobi F-relaxation is used with one more iteration than the degree of the Neumann approximation of
$R_{\textnormal{ideal}}$. Restriction is built using a degree-one Neumann expansion \eqref{eq:neumann} applied to strong
connections in $A_{ff}$, with connection drop-tolerance $\phi = 0.025$. These parameters are motivated through the
comparison of $n$AIR with $R_{\textnormal{ideal}}$ in Section \ref{sec:tri:air}.

From Theorem \ref{thm:WAPP}, if we approximate $R_{\textnormal{ideal}}$, then $P$ should be built targeting approximation
properties, in a classical AMG sense. \tcb{It turns out, classical AMG interpolation formulae \cite{Brandt:1985um,ruge:1987}
do not satisfy approximation properties on hyperbolic transport discretizations \cite{nonsymm}. Fortunately, $R$ is a very
accurate approximation to $R_{\textnormal{ideal}}$ for problems tested here, making interpolation less important.} We
propose a ``one-point interpolation'' scheme, where each F-point is interpolated by value from its strongest C-connection.
One-point interpolation resembles a degree-zero Neumann expansion, but $P$ is sparser, having exactly one nonzero per
row, and each nonzero is set to one as opposed to the value of $A_{fc}$.\footnote{One-point interpolation also resembles
an unsmoothed aggregation interpolation operator, but here some ``aggregates'' may have many points and others only
one. In aggregation-based AMG, aggregates are typically chosen to be comparable in size, and there are never aggregates
of size one that have strong connections in the matrix.} This ensures that the constant is in the range of interpolation, with
minimal nonzero requirements of $P$. In practice, one-point interpolation performs \tcb{best compared to} many different
interpolation methods tested in terms of total work and time to solution. \tcb{In fact, $\ell$AIR was shown to have good 
approximation properties for scalar hyperbolic problems in \cite{nonsymm}, and using these ideas to develop improved
interpolation methods for hyperbolic problems is ongoing work.}

\subsection{Blocks, filtering, and parallelization}\label{sec:alg:extra}

Some PDE discretizations lead to matrix equations with a natural block structure. The two most common examples are: (i) systems of
PDEs, where a block in the matrix corresponds to multiple variables discretized on a single spatial node, or (ii) block discontinuous
discretizations, such as discontinuous Galerkin (DG), where each finite element forms a block in the matrix. In either case, a block lower
triangular matrix can be transformed to lower triangular by scaling the system by the block-diagonal inverse, or $n$AIR can be performed
in the block setting, coarsening and forming transfer operators by block. For most results, we scale by the block-diagonal inverse because
the two approaches have shown comparable convergence factors, and forming $n$AIR in the setup phase is cheaper and simpler in the
scalar (non block) matrix case. However, Section \ref{sec:results:scale} shows results for block $n$AIR as well, where coarsening,
restriction, interpolation, and relaxation are all done in a block fashion.

One way to further reduce complexity in an AMG solver is truncating or lumping operators. The idea is simple:
remove entries from a matrix in the hierarchy, $A_\ell$, that are smaller than some threshold, typically with respect to the diagonal
element of the given row. Such methods have been used in AMG for symmetric problems with diffusive components \cite{Bienz:2015ve,
Falgout:2014uz,Treister:2015cp}). Heuristically, eliminating small entries is even more appropriate in the hyperbolic setting, because
the solution at any given point only depends on the solution at other points upwind along the characteristic. In the discrete setting,
small entries that arise in matrix operations are often not aligned with the characteristic and are more of a numerical effect, suggesting
that some can be eliminated without degrading convergence. Numerical results confirm this; in particular, removing
entries in the case of SPD matrices is a delicate process \cite{Bienz:2015ve,Falgout:2014uz,Treister:2015cp}, but Section \ref{sec:results}
shows that entries can be removed from discretizations of steady state transport aggressively, without a degradation in convergence.  
For some drop-tolerance $\varphi$, elements $\{a_{ij} \text{ $|$ } j\neq i, |a_{ij}| \leq \varphi|a_{ii}|\}$ are eliminated (that is, set to zero)
for each row $i$ of matrix $A_\ell$. A drop tolerance of $\varphi=10^{-3}$ has proved to be effective for many problems tested, and is
used for all results presented here. 

Finally, $n$AIR applied to triangular systems is intended for highly parallel environments, where traditional triangular solves are not easily
parallelized. This work focuses on algorithmic details and theory, and does not develop parallel performance models or present
parallel scaling results. However, it is well known that AMG scales in parallel to hundreds of thousands of cores \cite{Baker:2012ko},
with a communication cost of $O(\log P)$, for $P$ processors \cite{Falgout:2000hs,Henson:2002vk}. The algorithm developed here takes
on the form of a traditional AMG
method, with the additional cost of building and storing a restriction operator, which can be performed efficiently in parallel.

\section{Numerical results}\label{sec:results}

In this section, we apply $n$AIR to discontinuous discretizations of the steady state transport equation. For problems in which the
steady state is well posed (no cycles in flow), the steady state case is equivalent to the time dependent problem with an infinite time step.
In this context,
successful results on steady state flow also indicate that $n$AIR is applicable in the time-dependent regime with implicit time-stepping
schemes of arbitrary step size. 

The computational cost or complexity of an AMG algorithm is typically measured in \textit{work units} (WU),
where one WU is the cost to perform one sparse matrix-vector multiplication with the initial matrix. Operator complexity 
(OC) gives the cost in WUs to perform one sparse matrix-vector multiplication on each level in an AMG hierarchy, and
cycle complexity (CC) gives the cost in WUs to perform one AMG iteration, including pre- and post-relaxation,
computing and restricting the residual, and coarse-grid correction. For convergence factor $\rho$, the
\textit{work-unit-per-digit-of-accuracy} (WPD), $\chi_{wpd}$, is an objective measure of AMG performance, giving the total
WUs necessary to achieve an order-of-magnitude reduction in the residual:\footnote{Although WPD is a good
measure of serial performance, it does not reflect parallel efficiency of the algorithm.}
\begin{align*}
\chi_{wpd} := -\frac{CC}{\log_{10}(\rho)}  := -\frac{1}{\log_{10}(\rho_{eff})},
\end{align*}
{where $\rho_{eff} = \rho^{1/CC}$ is the {\it effective convergence factor}.}
\
\subsection{Test problems and discretizations} \label{sec:results:disc}

The model problem used here is the steady state transport equation:
\begin{align}\label{eq:transport}
\begin{split}
\mathbf{b}(x,y) \cdot\nabla u + c(x,y)u & = q(x,y) \hspace{3ex}{\cal D}, \\
u & = g(x,y) \hspace{3ex}\Gamma_{\textnormal{in}},
\end{split}
\end{align}
for domain ${\cal D}$ and inflow boundary $\Gamma_{\textnormal{in}}$. 
Multiple cases are studied that encompass spatially dependent source terms, $q(x,y)$,
discontinuities in the material coefficient, $c(x,y)$, and constant and non-constant flow direction, $\mathbf{b}(x,y)$, over
structured and unstructured meshes. \tcb{When $\mathbf{b}(x,y)$ is constant, we denote $\mathbf{b}(x,y) = \Omega(\theta) :=
(\cos(\theta),\sin(\theta))$, for some angle $\theta$.}
Two model domains are considered, the \textit{inset} domain and \textit{block-source}
domain, shown with solutions in Figure \ref{fig:domains}. In each domain, inflow boundaries consist of the south and west
boundaries with inflow $u=1$, and the material coefficient $c(x,y)$ is piecewise constant in both cases, with changes
of eight orders of magnitude, as shown in Figure \ref{fig:domains}. The \textit{inset} domain has no source ($q = 0$), while
the \textit{block-source} domain has an interior source $q(x,y) = 1$ in the interior block. 
These terms ($c(x,y)$ and $q(x,y)$) are fixed for all experiments, except the non-triangular case (Section
\ref{sec:results:extra}). Multiple velocity fields $\mathbf{b}(x,y)$ are considered. Solutions with constant flow are
shown in Figure \ref{fig:domains}, and several variations of the \textit{inset} domain with non-constant flow are shown in
Figure \ref{fig:flows} in Section \ref{sec:results:ang}. All numerical experiments use $c(x,y)$ and $q(x,y)$ as
specified in Figure \ref{fig:domains}, but multiple velocity fields $\mathbf{b}(x,y)$ are considered.

\begin{figure}[!ht]
  \centering
    \begin{subfigure}[b]{0.49\textwidth}
      \begin{subfigure}[b]{0.475\textwidth}
        \includegraphics[width=\textwidth]{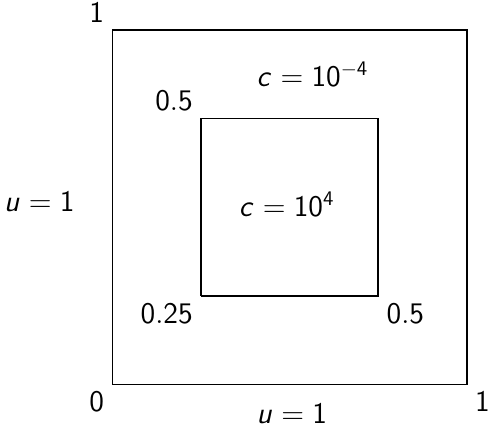}
      \end{subfigure}
      \hfill
       \begin{subfigure}[b]{0.475\textwidth}
        \includegraphics[width=\textwidth]{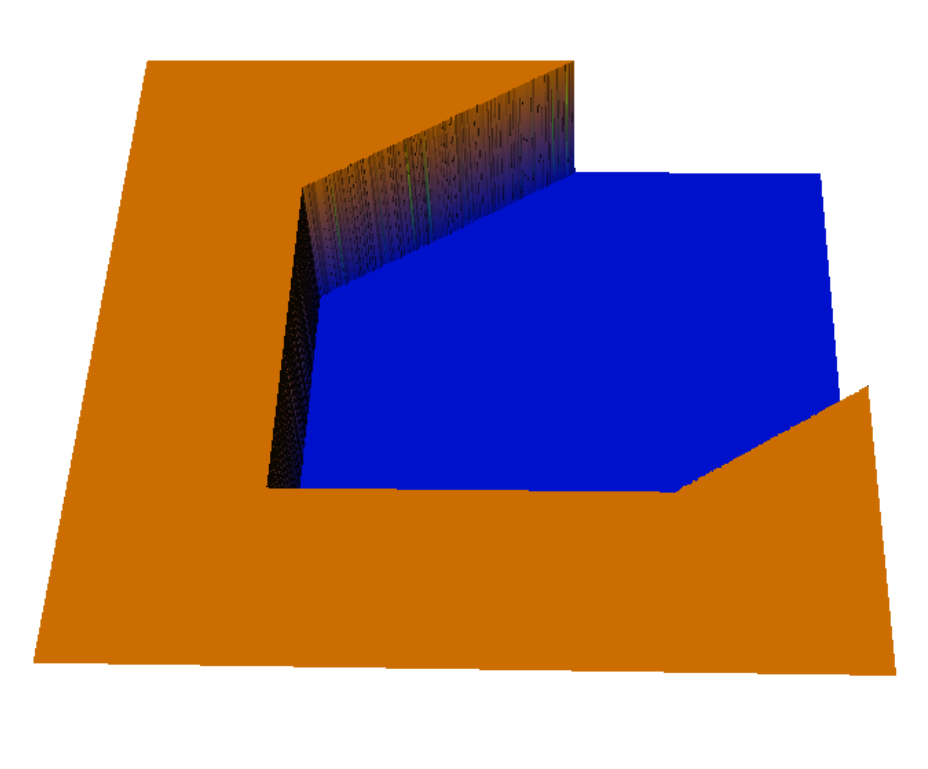}
      \end{subfigure}
  \caption{\textit{Inset} domain}
  \label{fig:inset}
  \end{subfigure}
    \begin{subfigure}[b]{0.49\textwidth}
      \begin{subfigure}[b]{0.475\textwidth}
        \includegraphics[width=\textwidth]{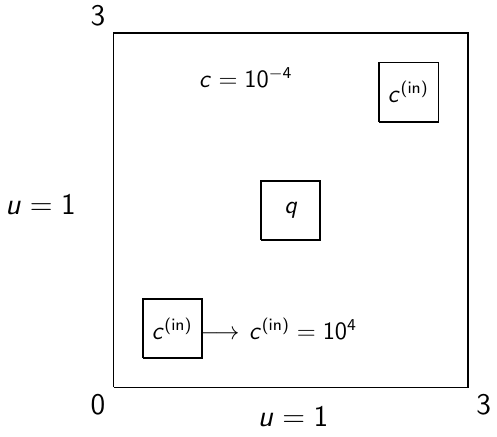}
      \end{subfigure}
      \hfill
       \begin{subfigure}[b]{0.475\textwidth}
        \includegraphics[width=\textwidth]{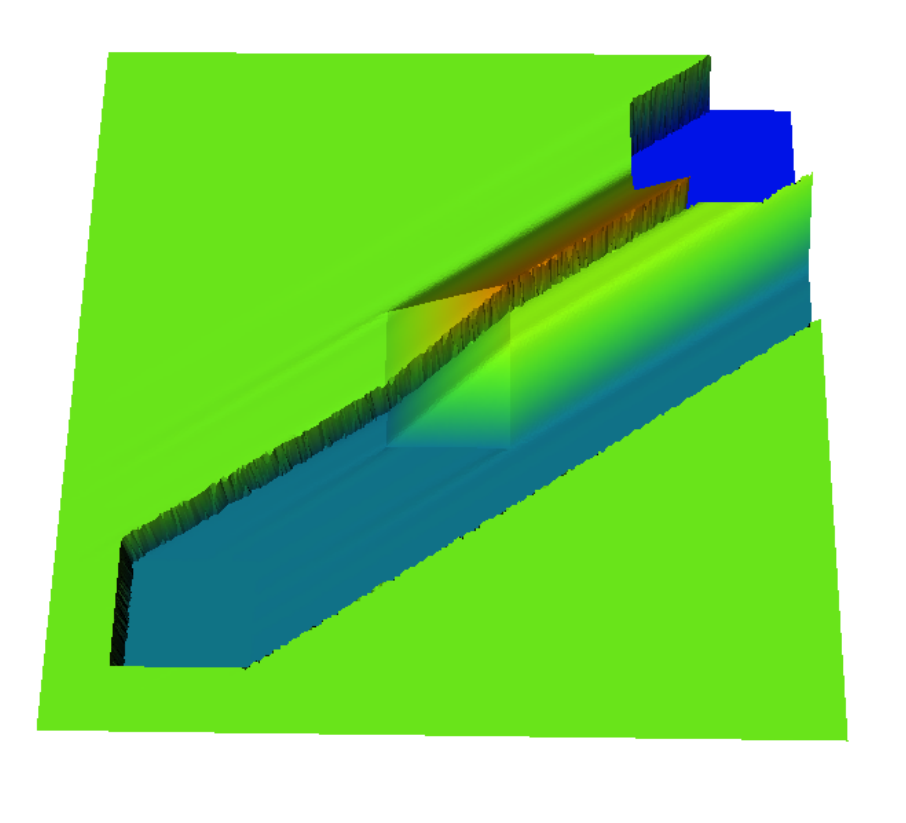}
      \end{subfigure}
  \caption{\textit{Block-source} domain}
  \label{fig:blocksource}
  \end{subfigure}
    \caption{Two domains for the steady state transport equation, with a constant velocity field,
  $\mathbf{b}(x,y) = (\cos(\theta),\sin(\theta))$, $\theta = 3\pi/16$, source $q = 1$, and the respective solutions.}
  \label{fig:domains}
  \end{figure}

To accompany the different domains considered, multiple upwind discretizations are implemented. A first-order lumped corner balance (LCB)
finite element
discretization \cite{Morel:2005tv,Morel:2007jj} is applied on structured and unstructured meshes. Standard fully upwinded discontinuous
Galerkin (DG) discretizations \cite{lesaint_raviart_1974,reed_hill_1973} are also tested, with finite element orders $1-6$.
A comprehensive introduction can be found in \cite{pietro_ern_2011}. Standard upwinded DG
methods arise as special cases in \cite{Brezzi:2004hf} and for almost-scattering-free problems in \cite{Ragusa:2012gn}.
The structured meshes used are triangular crossed-square meshes, conforming to the material discontinuities in $c(x,y)$, while
random triangulations, again conforming to material discontinuities, are used as unstructured meshes. Additional discretizations
based on highly elongated meshes with curvilinear elements, as well as continuous (linear) elements with artificial diffusion, are
briefly explored in Section \ref{sec:results:extra}.

As motivation for $n$AIR, we first highlight the difficulties that existing varieties of AMG face solving these discretizations. Tests
were run using the PyAMG library \cite{Bell:2008}, \textit{hypre} \cite{Falgout:2002vu}, and ML \cite{ml}. Classical AMG methods are not
well developed for the nonsymmetric setting; methods such as BoomerAMG in \textit{hypre}
\cite{Henson:2002vk} use a Galerkin, $P^TAP$ coarse grid, and are able to solve discontinuous transport discretizations based on linear finite
elements, with convergence factors on the order of 0.8--0.9. However, convergence is not scalable and does not extend beyond
linear elements. Aggregation- and energy-minimization-based AMG methods are better developed for nonsymmetric problems. The most
successful existing solver appears to be the nonsymmetric smoothed aggregation (NSSA) solver in the ML library \cite{ml,Sala:2008cv,Wiesner:2014cy}.
With GMRES acceleration, NSSA is able to converge on most problems tested here. In all cases, however, NSSA takes several times more
iterations than $n$AIR and often requires significant relaxation, such as a $V(3,3)$-cycle, for good convergence. For difficult problems, $n$AIR offers
a speedup of $5\times$ or more over the current state-of-the-art.

\subsection{Angular variation, non-constant flow, and 3d} \label{sec:results:ang}

Problems in higher dimensions, anisotropies on unstructured meshes, and non-grid-aligned anisotropies can
prove difficult for AMG solvers \cite{Manteuffel:2017,Schroder:2012di}. Here, we show $n$AIR to be robust in all of
these cases. Figure \ref{fig:angular} shows performance of $n$AIR for LCB discretizations of the \textit{inset} problem
on structured and unstructured meshes, with fixed angle $\Omega := \mathbf{b}(x,y) = (\cos(\theta),\sin(\theta))$,
and angles $\theta\in[0,\sfrac{\pi}{2}]$. Because unstructured meshes are often used in practice and typically
more difficult from a solver perspective, further results use an unstructured mesh and the angle is (arbitrarily) fixed to
$\theta = \sfrac{3\pi}{16}$. \footnote{For some angles, $n$AIR converges faster on an unstructured mesh than a structured mesh. However, 
the wall-clock time of the setup and solve phase is at least $2\times$ faster in \textit{all cases} for a structured mesh
over an unstructured mesh. It is possible that a structured mesh makes for a more structured matrix amenable to matrix-vector
operations, but a detailed analysis is outside the scope of this work.}

\begin{figure}[!th]
  \centering
  \begin{subfigure}[b]{0.425\textwidth}
    \includegraphics[width=\textwidth]{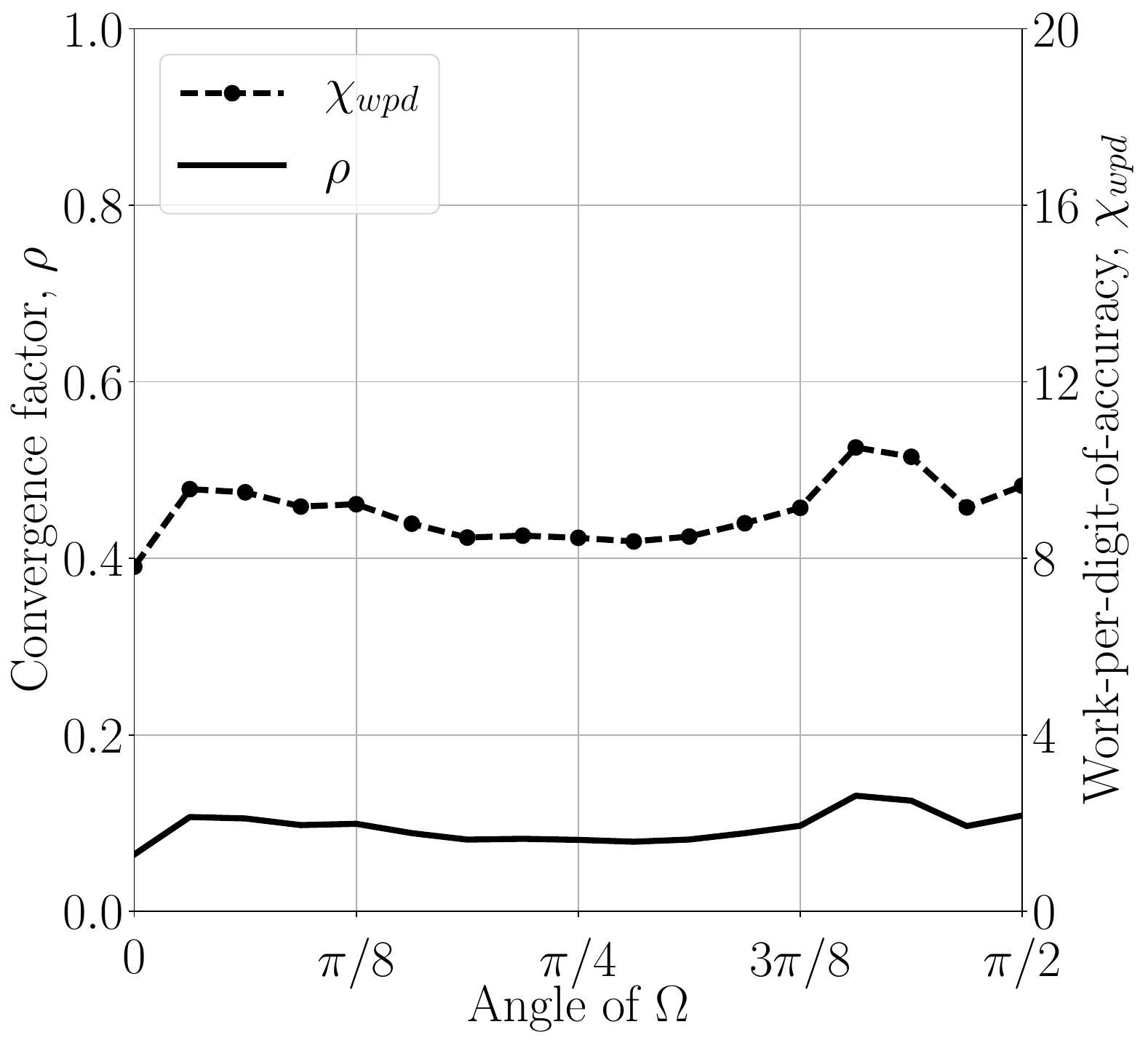}
    \caption{LCB unstructured mesh}\label{fig:angular:unstruct}
  \end{subfigure}
  \hspace{2ex}
    \begin{subfigure}[b]{0.425\textwidth}
    \includegraphics[width=\textwidth]{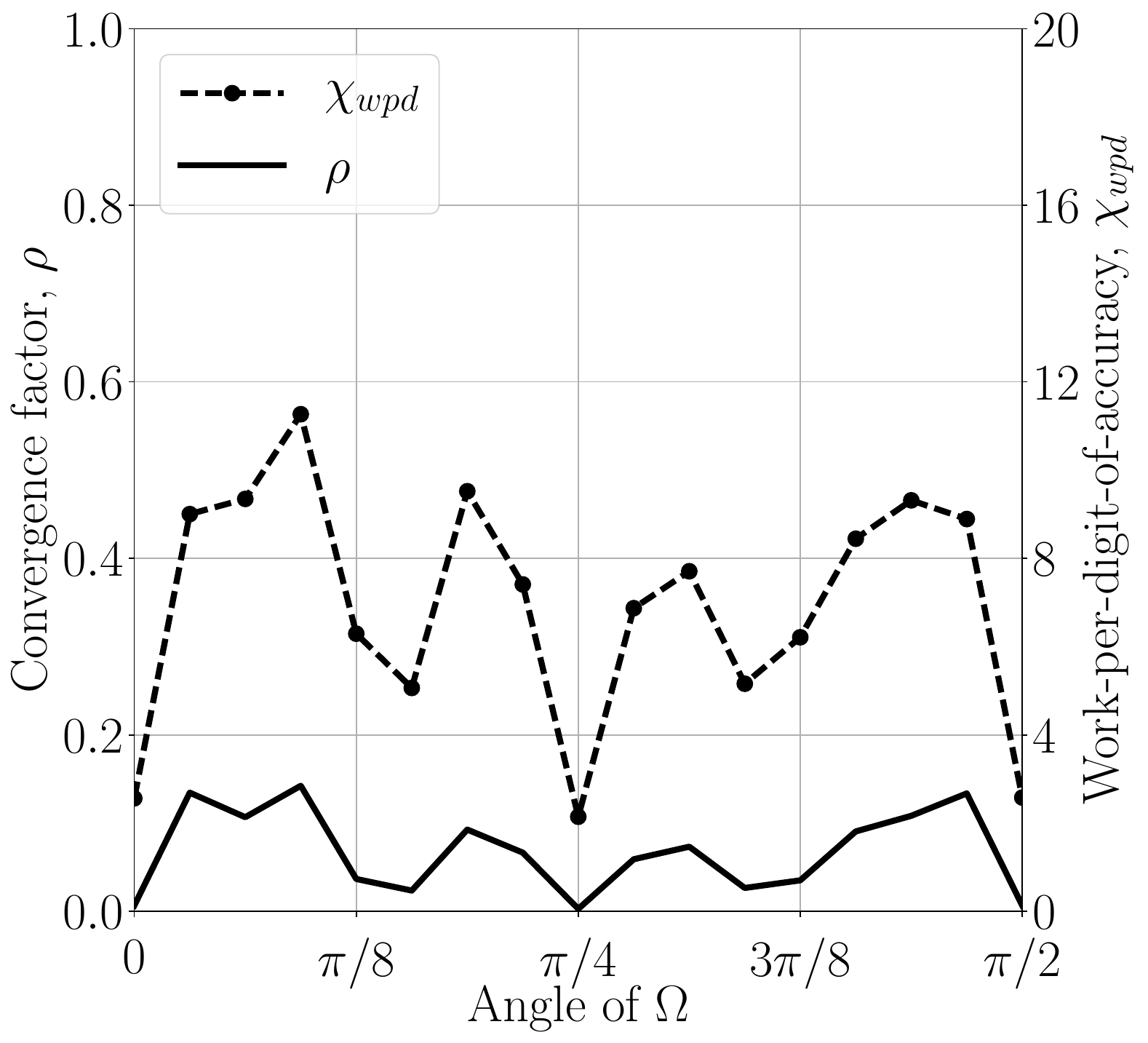}
    \caption{LCB structured mesh}\label{fig:angular:struct}
  \end{subfigure}
  \caption{Convergence factors and WPD for $n$AIR applied to
  LCB discretizations of the \textit{inset} problem, with angles between $0$ and $\sfrac{\pi}{2}$, on
  unstructured and structured meshes, and $\approx 2.25$M DOFs.}
  \label{fig:angular}
\end{figure}

In addition to being robust with respect to angular variations, $n$AIR is insensitive to
flow direction and problem dimensionality. Figure \ref{fig:flows} shows the solution to three different
non-constant flows defined on the \textit{inset} domain, and the corresponding performance of $n$AIR, along with results for
a fixed direction on the inset and block-source domain.

Table \ref{tab:3d} shows results of $n$AIR applied to steady state transport in three dimensions for different finite element orders.
The three-dimensional domain is a unit cube, with $c(x,y) = 10^4$ inside of a centered interior cube of size
$0.5\times 0.5\times 0.5$, and $c(x,y) = 10^{-4}$ outside of that subdomain, similar to the 2d \textit{inset} domain
(Figure \ref{fig:domains}). A random tetrahedral mesh is used, conforming to discontinuities in $c(x,y)$, and a constant
velocity field, $\mathbf{b}(x,y) = \left(\sin(\theta_1)\cos(\theta_2),\sin(\theta_1)\sin(\theta_2),\cos(\theta_1)\cos(\theta_2)\right)$.
As in 2d, choice of $\theta_1$ and $\theta_2$ does not affect results on an unstructured mesh (see Figure \ref{fig:angular});
here we use $\theta_1=\theta_2 = 3\pi/16$.
In all cases, $n$AIR is able to achieve fast convergence at a moderate cost. Due to the increased matrix connectivity in three dimensions,
filtering is particularly useful here, reducing WPD by a factor of four or more in all cases, and total time-to-solution (not shown) by factors of 3--4.

This is also a good example to demonstrate the speedup that $n$AIR can provide over $\ell$AIR in setup time. For 3rd-order
elements in 3D, with about 2M DOFs, distance-1 $\ell$AIR takes 2995 seconds to build the solver and 43 seconds to solve to
$10^{-12}$ residual tolerance. Distance-one $n$AIR (approximately corresponding to distance-two $\ell$AIR) takes only 38
second to setup and 52 seconds to solve. Despite convergence factors that are slightly larger with $n$AIR compared to $\ell$AIR,
the overall time to solution is significantly smaller. Furthermore, it is simple and moderately cheap to go from distance-one to
distance-two $n$AIR; for example, the setup time here increases modestly to about 60 seconds with distance-two $n$AIR, while
distance-three $\ell$AIR (approximately corresponding to distance-two $n$AIR) is completely intractable. 

\begin{figure}[!h]
  \centering
  \begin{subfigure}[b]{0.25\textwidth}
    \includegraphics[width=\textwidth]{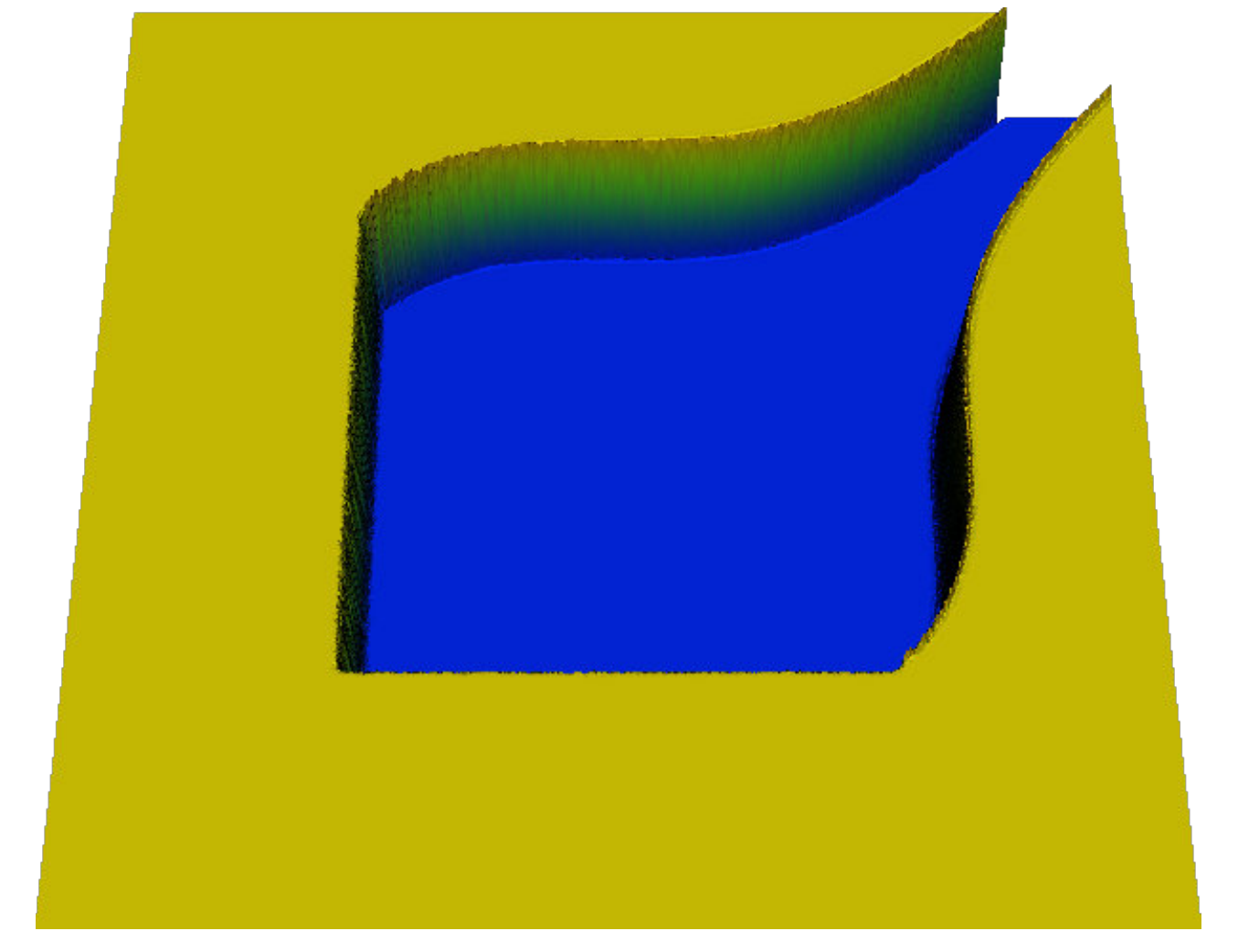}
    \caption{\centering$\mathbf{b}_1(x,y) =$ $(\cos(\pi y)^2, \cos(\pi x)^2)$.}
    \vspace{3ex} 
  \end{subfigure}
  \hspace{1ex}
  \begin{subfigure}[b]{0.25\textwidth}
    \includegraphics[width=\textwidth]{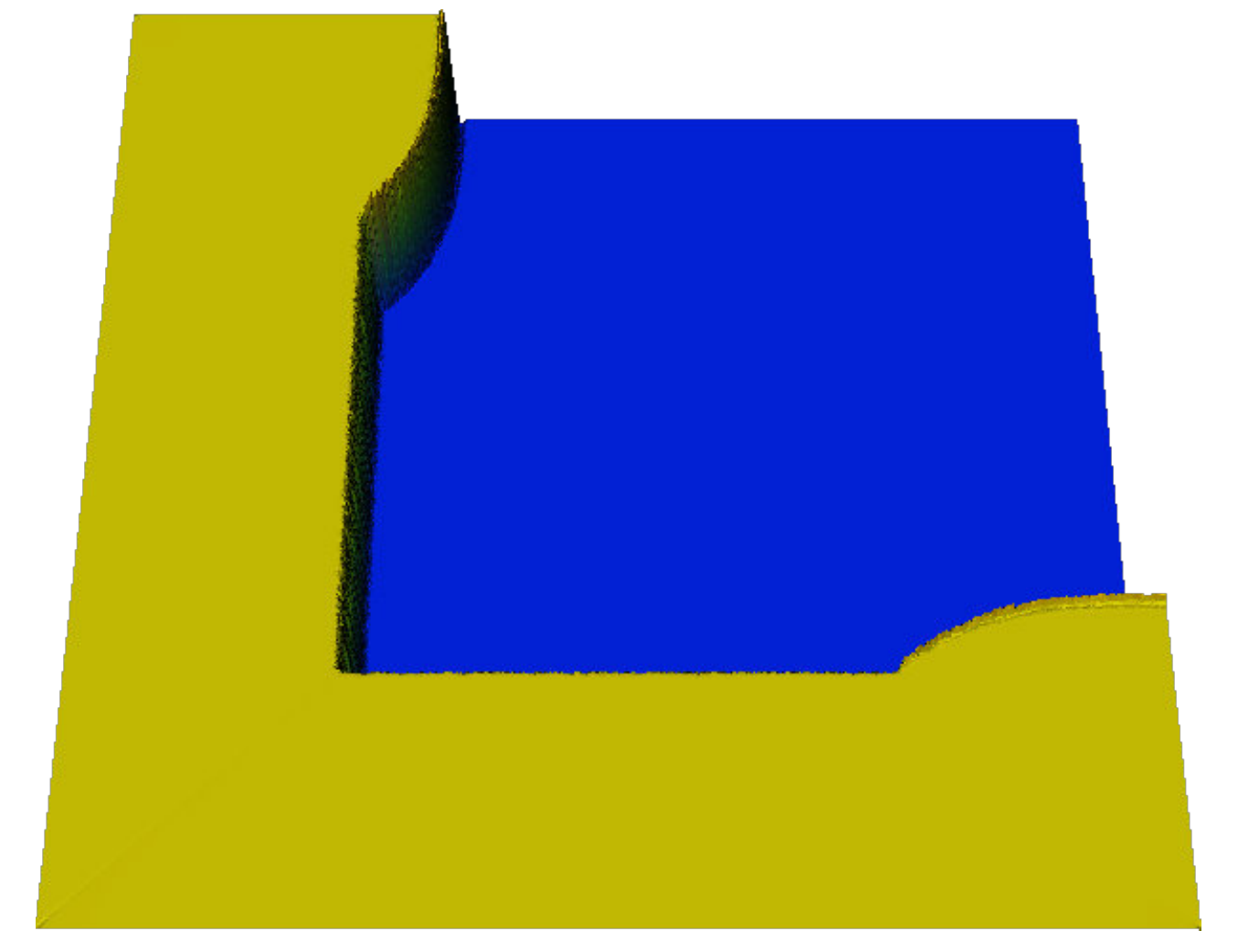}
    \caption{\centering$\mathbf{b}_2(x,y) =$ $(\sin(\pi y)^2, \sin(\pi x)^2)$.}
    \vspace{3ex} 
  \end{subfigure}
  \hspace{1ex}
  \begin{subfigure}[b]{0.25\textwidth}
    \includegraphics[width=\textwidth]{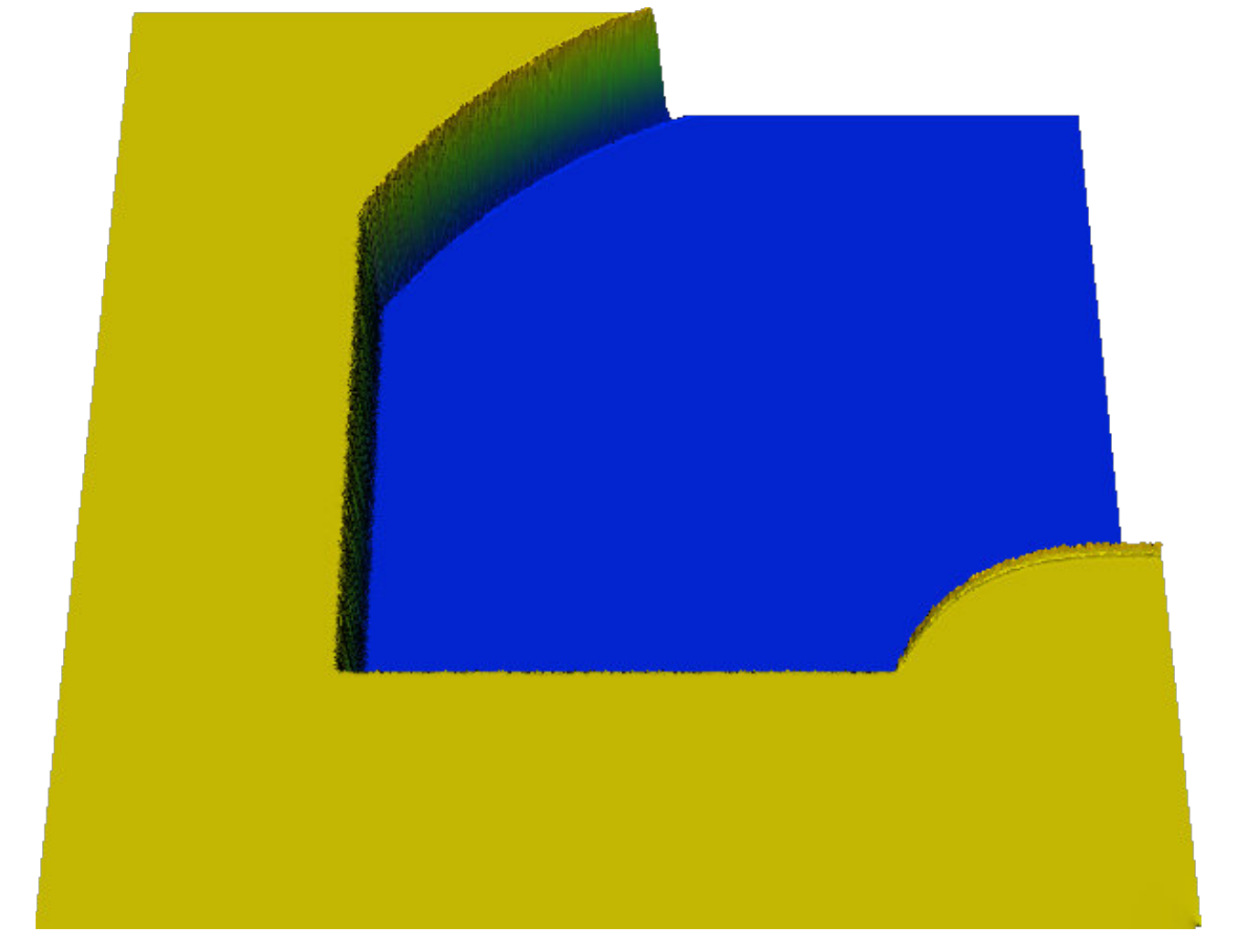}
    \caption{\centering$\mathbf{b}_3(x,y) =$ $(y^4, \cos(\sfrac{\pi x}{2})^2)$.}
    \vspace{3ex} 
  \end{subfigure}
{
\renewcommand{\tabcolsep}{5pt}
\begin{tabular}{|c | c c c c c |}\toprule
$\mathbf{b}(x,y)$ & $\Omega_{\textnormal{inset}}$ & $\mathbf{b}_1(x,y)$ & $\mathbf{b}_2(x,y)$ & $\mathbf{b}_3(x,y)$ & $\Omega_{\textnormal{block-source}}$ \\\midrule
$\rho$ 				  & 0.20 & 0.26 & 0.17 & 0.24 & 0.25\\
CC					  & 6.83 & 6.77 & 6.80 & 6.81 & 7.53\\
$\chi_{\textnormal{WPD}}$ & 9.68 & 11.46 & 8.89 & 10.89 & 12.36 \\\bottomrule
\end{tabular} 
}
\caption{Convergence factor, CC, and WPD, for $n$AIR applied to variations in flow direction, $\mathbf{b}(x,y)$, on the inset domain, and constant flow direction
	$\Omega =(\cos(\sfrac{3\pi}{16}),\sin(\sfrac{3\pi}{16}))$ on both domains. All discretizations have $\approx 9$M DOFs. }
\label{fig:flows}
\end{figure}

\begin{table}[!ht]
\color{black}
\centering
\begin{tabular}{|c c c c r r r c |}\toprule
degree FEM & $n$ & nnz & $\rho$ & CC & $\chi_{\textnormal{WPD}}$ & $\varphi$ & Speedup\\ \midrule
1 & 2.5M & 24M & 0.09 & 42.4 & 41.7 & 0 & -- \\
1 & 2.5M & 24M & 0.10 & 10.9 & 11.0 & 1e-3 & 3.9 \\\midrule
2  & 1.9M & 41M & 0.12 & 38.4 & 42.8 & 0 & -- \\
2  & 1.9M & 41M & 0.13 & 8.5 & 9.9 & 1e-3 & 4.3 \\\midrule
3 & 2.2M & 85M & 0.14 & 32.6 & 38.2 & 0 & -- \\
3 & 2.2M & 85M & 0.17 & 6.8 & 8.9 & 1e-3 & 4.3 \\\bottomrule 
\end{tabular} 
\caption{(3D) $n$AIR applied to first-, second-, and third-order discretizations of steady state transport in three dimensions.
The final column shows the speedup due to filtering in terms of WPD. The rows in each block differ in drop tolerance, $\varphi$. }
\label{tab:3d}
\end{table}

\subsection{Scaling in $h$ and element order}\label{sec:results:scale}

Next, we study the scaling of $n$AIR with respect to DOFs. One exciting feature of
$n$AIR is its ability to solve high-order finite-element discretizations, something that AMG methods often struggle with. 
Here, V-cycles and F-cycles are considered, with degree of $n$AIR $k=1,2$ and $3$. Although F-cycles originate in full multigrid, which focuses
on achieving discretization-level accuracy in a single F-cycle \cite{BrHeMc2000}, accuracy with respect to discretization is not considered in this work.
Instead, the F-cycle is used as it can provide more robust convergence and scaling than a V-cycle, at a much lower cost than alternatives
such as W-cycles and K-cycles. Figure \ref{fig:scale} shows scaling of WPD and convergence factor of $n$AIR applied to upwind
DG discretizations of the inset problem, as a function of number of DOFs.
Although there remains a slow growth in WPD, likely due to an increase in iterations before asymptotic convergence rates are achieved,
convergence factors have effectively asymptoted for lower-order finite elements, and are leveling off even for 6th-order finite elements.  

\begin{figure}[!htb]
  \centering
  \begin{subfigure}[b]{0.45\textwidth}
    \includegraphics[width=\textwidth]{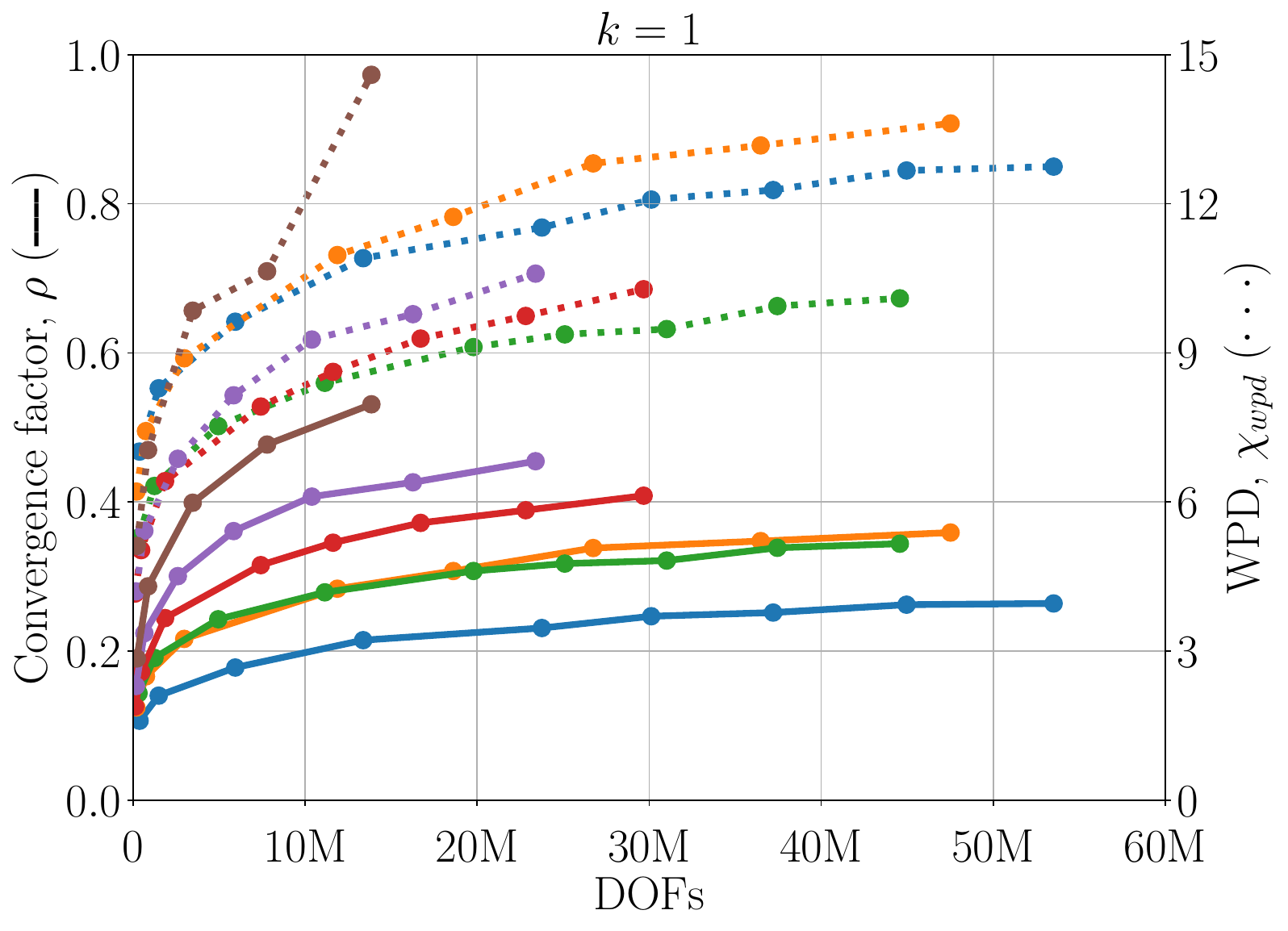}
    \includegraphics[width=\textwidth]{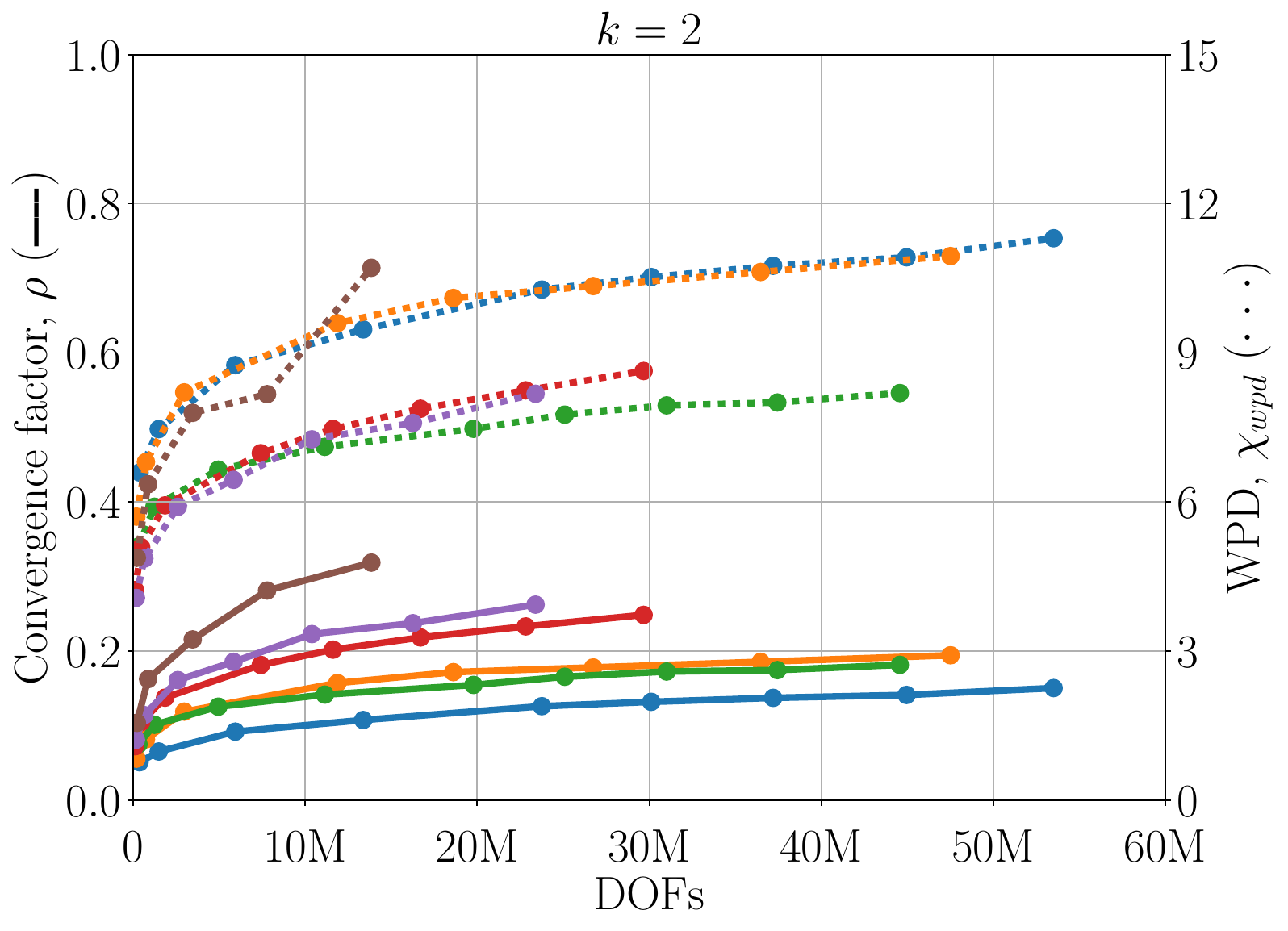}
    \includegraphics[width=\textwidth]{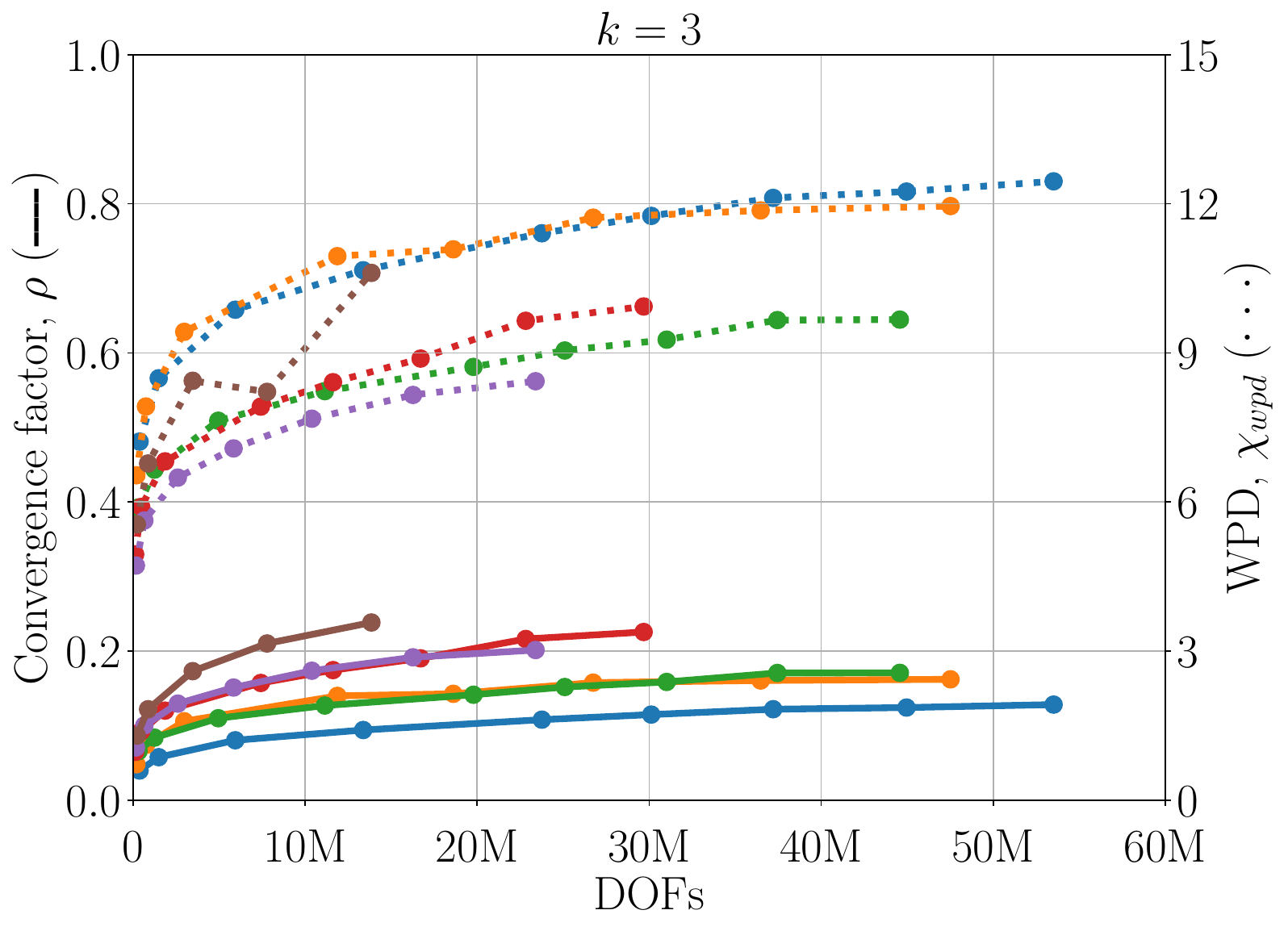}    
    \caption{V-cycle}
  \end{subfigure}
  \hfill
  \begin{subfigure}[b]{0.45\textwidth}
    \includegraphics[width=\textwidth]{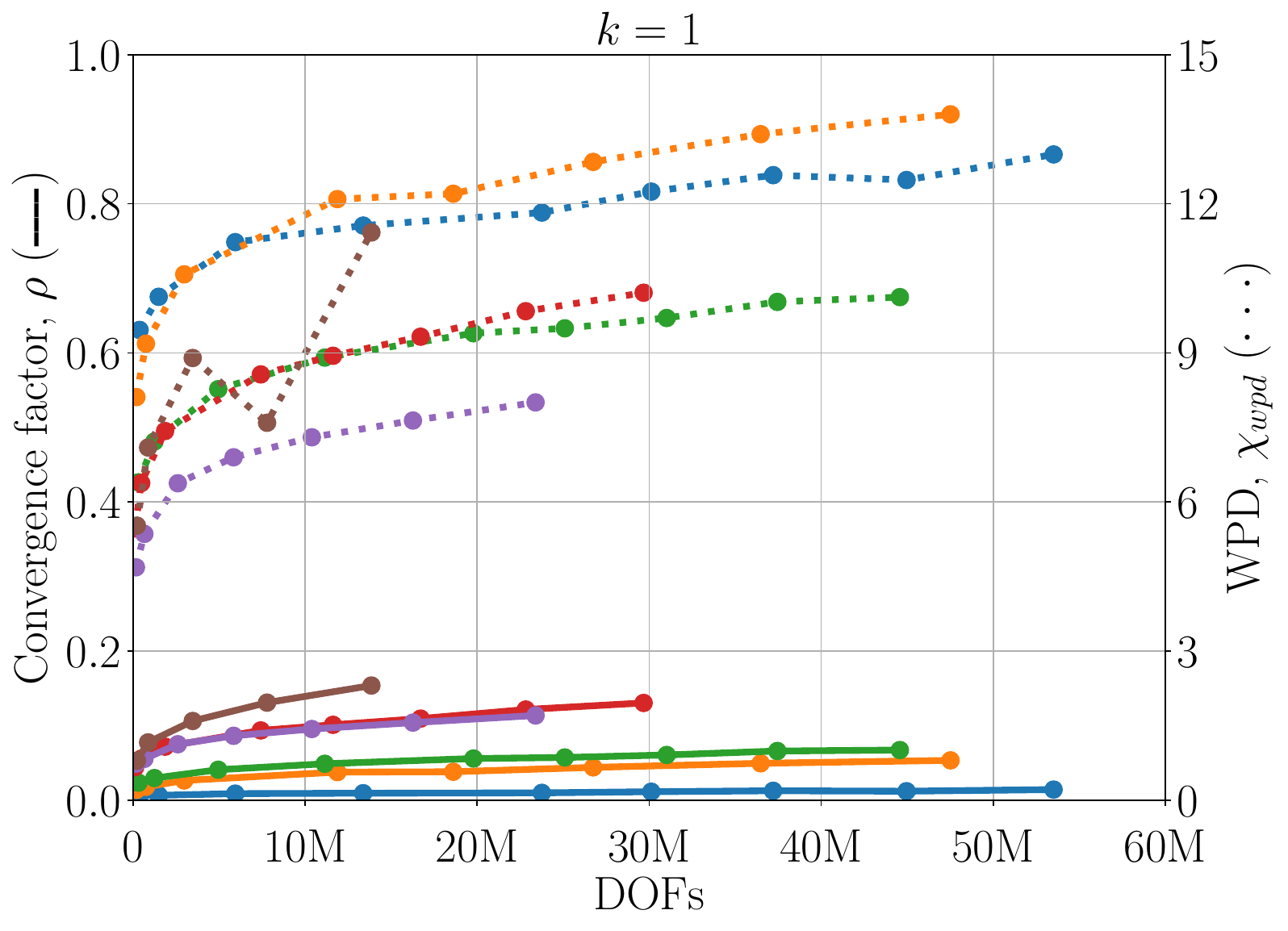}
    \includegraphics[width=\textwidth]{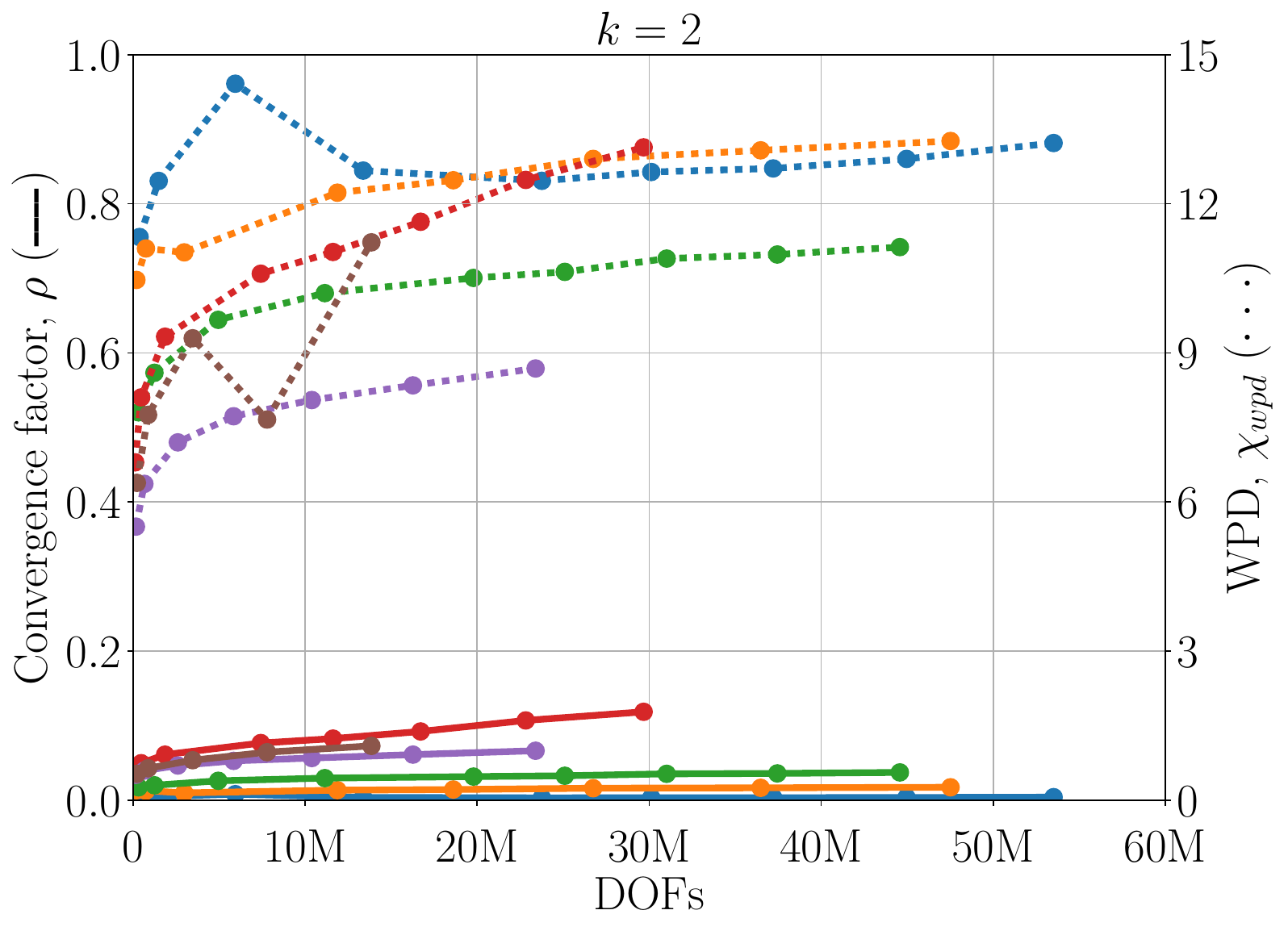}
    \includegraphics[width=\textwidth]{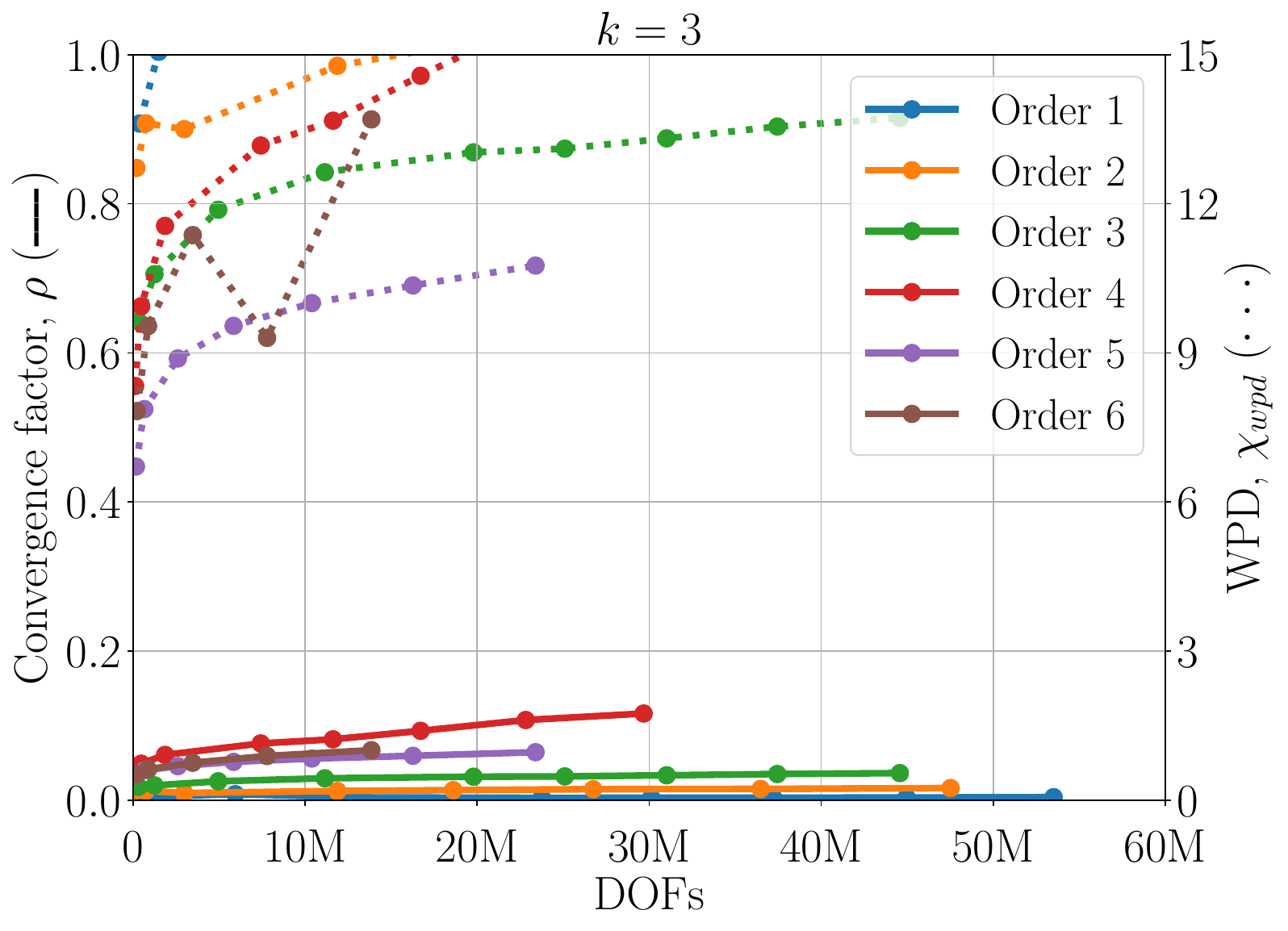}    
    \caption{F-cycle}
  \end{subfigure}
  \caption{Scaling of convergence factor (solid lines) and WPD (dotted lines) as a function of DOFs,
  for $n$AIR applied to upwind DG discretizations of the inset problem, with $n$AIR degrees $k=1,2$ and $3$, 
  V-cycles and F-cycles, and finite-element degrees 1--6. }
  \label{fig:scale}
\end{figure}

One benefit of the reduction approach is that convergence of $n$AIR can be improved by increasing the accuracy of the Neumann
approximation. Figure \ref{fig:scale} demonstrates that convergence is improved a notable amount by increasing $k=1$ to
$k=2$, and again for $k=2$ to $k=3$. This allows us to attain V-cycle convergence factors on the order of $\rho\approx0.2$ for
6th-order finite elements. Similar improvements in convergence can be obtained by decreasing the strength tolerance, $\phi$,
leading to a more accurate approximation of $A_{ff}^{-1}$. Because increasing $k$ also increase the density of coarse-grid
operators, in practice, one must find the balance between complexity and convergence factors. Here, we see that the V-cycle
with $k=2$ appears to be the most effective choice, \tcb{because the WPD is less than that of $k=1,3$, or F-cycles for all
finite element orders. The best choice in terms of total time (including setup) likely depends on how many linear
systems are being solved, and it is possible $k=1$ is faster for a single or small number of systems.} Note that having the option to make
this choice is a benefit of AIR, because, in general, classical AMG methods do not have a natural and robust way to
improve convergence the way that increasing $k$ can. 

So far all results have employed scaling the matrix by the block-diagonal inverse. However, $n$AIR is also amenable to a
block implementation, where coarsening, restriction, and interpolation are all done by block. This is particularly relevant for
systems of PDEs with coupled variables, where scaling out the block diagonal is unlikely to capture the necessary couplings.
Figure \ref{fig:blockscale} demonstrates block $n$AIR applied to the same problem
as in  Figure \ref{fig:scale}, this time directly using the DG block structure, for orders 1, 2, and 3 finite elements. Convergence
factors are similar to those achieved by scaling by the block diagonal inverse. The WPD is higher because we do not
filter in the block setting, and because, when using a block sparse matrix, even zero entries in otherwise nonzero blocks must
be stored. This results in a larger OC and, thus, larger WPD. 

\begin{figure}[!htb]
  \centering
    \includegraphics[width=0.45\textwidth]{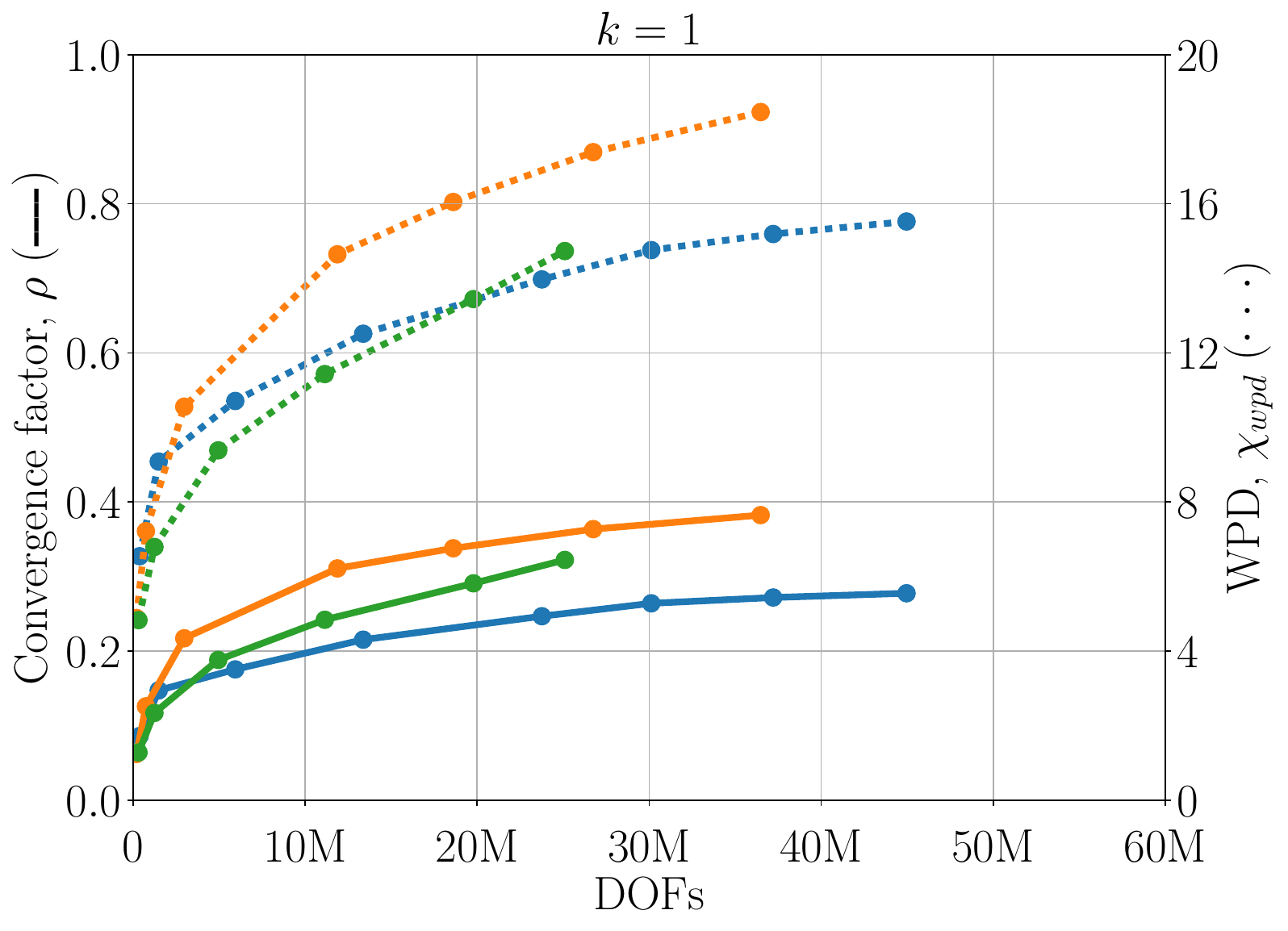}
    \hfill
    \includegraphics[width=0.45\textwidth]{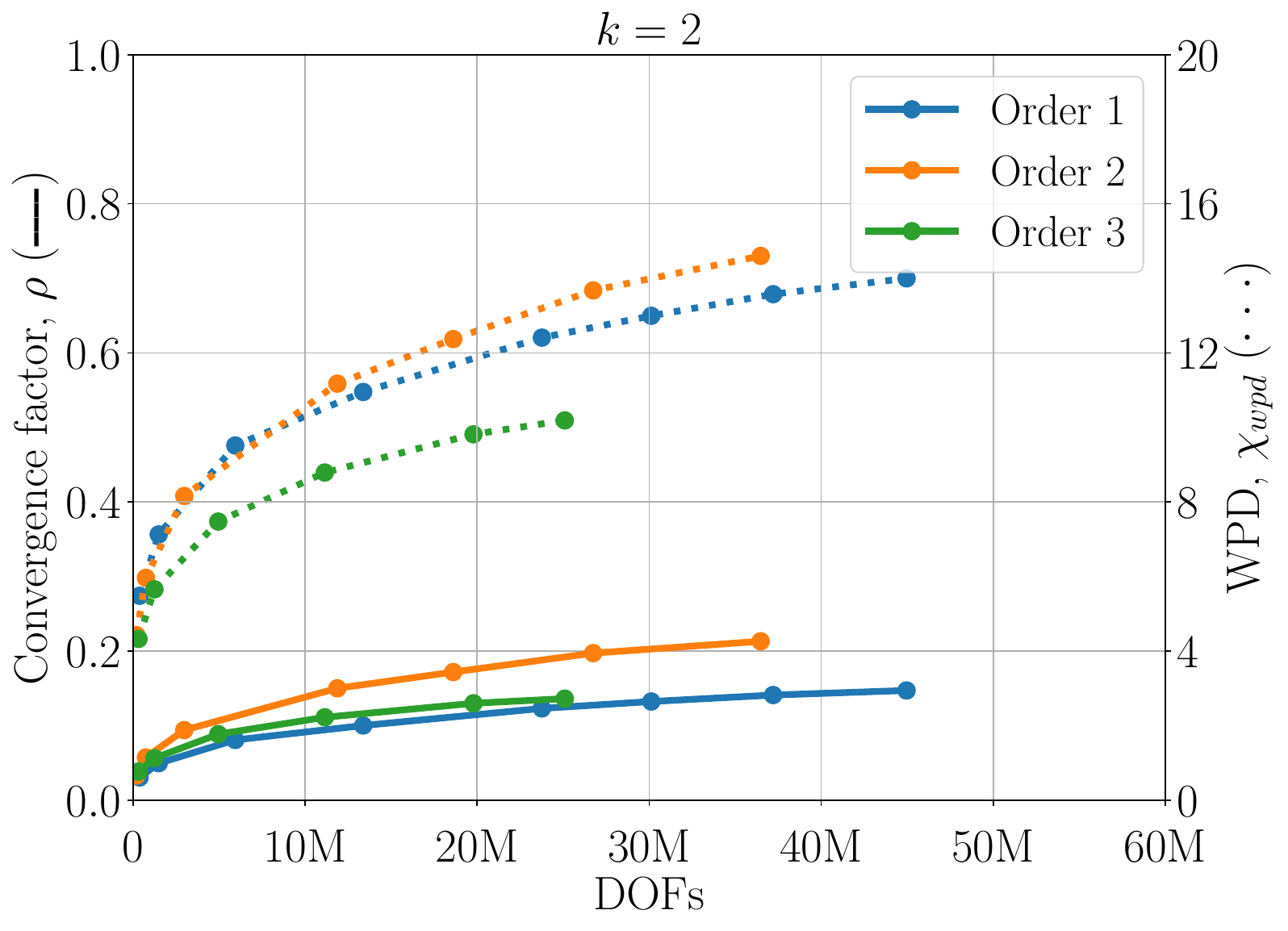}
  \caption{Scaling of convergence factor (solid lines) and WPD (dotted lines) as a function of DOFs,
  for block-$n$AIR applied to upwind DG discretizations of the inset problem, with $n$AIR degrees $k=2$ and $3$,
  V-cycles, finite-element degrees 1--3, and corresponding block sizes 3, 6, and 10.}
  \label{fig:blockscale}
\end{figure}

\subsection{Non-triangular matrices} \label{sec:results:extra}

Finally, it is well-known that a Neumann approximation for ideal operators is, in general, not effective for non-triangular matrices, such as a symmetric
discretization of diffusion. However, for ``nearly triangular'' matrices, $n$AIR remains an effective solver. Here, we demonstrate the performance
of $n$AIR on three such problems in Table \ref{tab:extra}.

{\color{black}
\textbf{P1} corresponds to a streamline upwind Petrov-Galerkin (SUPG) discretization
of \eqref{eq:transport} on the \textit{inset} problem, using linear finite elements. SUPG discretizations use an upwinded
scheme for advection, with a small diffusive component added for stability \cite{Brooks:1982bl}. This results in a small, global, symmetric
component added to a triangular matrix. \textbf{P2} and \textbf{P3} correspond to 4th-order, upwind DG discretizations of
\eqref{eq:transport} in 2d, on high-order curvilinear finite elements \cite{curve}. Here, $c(x,y) = 2xy + 2x^2 + 1.2$, $\mathbf{b}(x,y)
= (1/\sqrt{3}, 1/\sqrt{3})$, and $q(x,y)$ is the right-hand side corresponding
to the exact solution $u = (x^2 + y^2 + 1)/2 + \cos(2.5(x+y))  ({b}_0^2 + {b}_1)$.
Curvilinear finite elements can be non-convex and produce
cycles in the mesh, wherein the resulting matrix for a fixed direction of flow is mostly block triangular, with some number of
strongly connected components (SCCs) that are non-triangular and large in magnitude \cite{sweep18}.
\textbf{P2} has 97 SCCs of size two and 15 SCCs
of 3--6 DOFs. \textbf{P3} has 40 SCCs of size two, 11 small SCCs with 3--6 DOFs, and one large SCC consisting of 1951 DOFs,
implying there is a substantial non-triangular block in the matrix.

It should be noted that even in serial, solving such a problem algebraically is nontrivial, particularly when there is a global symmetric 
component. One approach is to use a lower-triangular preconditioner that inverts the advective components exactly, equivalent to an
ordered Gauss-Seidel iteration. This would likely converge well; however, without geometric knowledge of the velocity field and
corresponding ordering of DOFs in the matrix, the right relaxation ordering cannot be easily determined. In \cite{sweep18}, a
cycle-breaking strategy was proposed to approximate an optimal relaxation ordering in the context of larger transport simulations on 
meshes with curvilinear elements. This proved successful on linear systems with SCCs, but convergence of the larger transport ``source iteration''
was $2-3\times$ slower than using an exact solve, and this approach does not overcome the inherent limitation of Gauss-Seidel in
parallel. 

Each of these problems is nearly triangular in a different regard. Here, we apply one $n$AIR V-cycle as a preconditioner for
GMRES, and results in Table \ref{tab:extra} show $n$AIR to be an effective preconditioner in all cases. 
}

\begin{table}[!h]
\centering
\begin{tabular}{|c | r r r r r |}\toprule
Problem & $n$ & $k$ & $\rho$ & CC & $\chi_{\textnormal{WPD}}$ \\ \midrule
\textbf{P1} & 1553001 & 1 & 0.57 & 8.56 & 35.74 \\\midrule
\textbf{P2} & 88800 & 1 & 0.22 & 4.52 & 6.85 \\
\textbf{P2} & 88800 & 2 & 0.14 & 5.63 & 6.71 \\\midrule
\textbf{P3} & 88800 & 1 & 0.46 & 4.54 & 13.59 \\
\textbf{P3} & 88800 & 2 & 0.35 & 5.72 & 12.62 \\\bottomrule 
\end{tabular} 
\caption{Results of $n$AIR-preconditioned GMRES applied to the nearly-triangular discretizations \textbf{P1}, \textbf{P2}, and \textbf{P3},
corresponding to an SUPG discretization of the inset problem, and 4th-order DG discretizations of a variation in \eqref{eq:transport} on
two meshes with curvilinear elements, respectively (see text for details).}
\label{tab:extra}
\end{table}

Classical AMG with GMRES acceleration does not
converge on either problem. NSSA using GMRES acceleration and Jacobi relaxation converges for both
problems. For the harder problem, \textbf{P}$_3$, NSSA converges with factors on the order of $\rho = 0.78$ for a $V(2,2)$-cycle and
$\rho = 0.68$ for a $V(3,3)$-cycle, with cycle complexities likely on the order of 8-9 and 11, respectively, \tcb{and corresponding WPDs
of 66 and 79}.\footnote{ML does not detail complexity in the same manner as PyAMG, so these estimates are for comparison purposes
between NSSA and nAIR.} However, this convergence is
sensitive to parameters. For example, reasonable convergence is attained for an aggregation threshold of $\phi = 0.2$, but NSSA
does not converge for $\phi = 0.3$, and iterations are halted due to a singular GMRES Hessenberg matrix for $\phi = 0.1$.

\section{Conclusions}\label{sec:conclusions}

This work studies reduction-based AMG for highly nonsymmetric matrices. Theory is developed indicating that, along with accurate
approximations of the ideal operators, a scalable method also requires that interpolation or restriction satisfy a classical multigrid
approximation property. A reduction-based AMG method is then developed, denoted $n$AIR, which is shown to be an effective solver
for upwind discretizations. Strong convergence factors are shown when $n$AIR is applied to the steady state transport equation, on
multiple domains,  with high-order upwind discretizations, and unstructured meshes or meshes with curvilinear elements. Although $n$AIR as presented here
proves robust for several ``nearly triangular'' problems, when significant non-triangular components are introduced, the performance
of $n$AIR will quickly degrade, and the more general variation, $\ell$AIR \cite{air2}, is more appropriate.

A serial implementation of $n$AIR is available in the PyAMG library \cite{Bell:2008}, available at \url{https://github.com/ben-s-southworth/pyamg/tree/air},
and a parallel implementation is now available in the \textit{hypre} library \cite{Falgout:2002vu} (\url{https://github.com/hypre-space/hypre}). 

\section*{Acknowledgment} 
The authors thank Steve McCormick for helpful discussions in the development of $n$AIR. 

\bibliographystyle{siamplain}
\bibliography{air.bib}

\begin{thebibliography}{10}

\bibitem{Adams:2013to}
{\sc M.~P. Adams, M.~L. Adams, and W.~D. Hawkins}, {\em {Provably optimal
  parallel transport sweeps on regular grids}}, in International Conference on
  Mathematics and Computational Methods Applied to Nuclear Science {\&}
  Engineering, Idaho, 2013, Texas A and M University, College Station, United
  States, pp.~2535--2553.

\bibitem{Alvarado:1993bl}
{\sc F.~L. Alvarado and R.~Schreiber}, {\em {Optimal parallel solution of
  sparse triangular systems}}, SIAM Journal on Scientific Computing, 14 (1993),
  pp.~446--460.

\bibitem{Amarala:2013bx}
{\sc S.~Amarala and J.~W.~L. Wan}, {\em {Multigrid Methods for Systems of
  Hyperbolic Conservation Laws}}, Multiscale Modeling {\&} Simulation, 11
  (2013), pp.~586--614.

\bibitem{Anderson:1989}
{\sc E.~Anderson and Y.~Saad}, {\em Solving sparse triangular linear systems on
  parallel computers}, International Journal of High Speed Computing, 1 (1989),
  pp.~73--95.

\bibitem{Bailey:2009tk}
{\sc T.~S. Bailey and R.~D. Falgout}, {\em {Analysis of massively parallel
  discrete-ordinates transport sweep algorithms with collisions}}, in
  International Conference on Mathematics, Computational Methods, and Reactor
  Physics, Saratoga Springs, NY, 2009, Lawrence Livermore National Laboratory,
  Livermore, United States, pp.~1751--1765.

\bibitem{Baker:2012ko}
{\sc A.~H. Baker, R.~D. Falgout, T.~V. Kolev, and U.~M. Yang}, {\em {Scaling
  Hypre{\textquoteright}s Multigrid Solvers to 100,000 Cores}}, in
  High-Performance Scientific Computing, Springer London, London, 2012,
  pp.~261--279.

\bibitem{Bell:2008}
{\sc W.~N. Bell, L.~N. Olson, and J.~B. Schroder}, {\em {PyAMG}: Algebraic
  multigrid solvers in {Python} v3.0}, 2015, \url{https://github.com/pyamg/}.
\newblock Release 3.0.

\bibitem{benzi2005numerical}
{\sc M.~Benzi, G.~H. Golub, and J.~Liesen}, {\em Numerical solution of saddle
  point problems}, Acta numerica, 14 (2005), pp.~1--137.

\bibitem{Bienz:2015ve}
{\sc A.~Bienz, R.~D. Falgout, W.~Gropp, and L.~N. Olson}, {\em {Reducing
  Parallel Communication in Algebraic Multigrid through Sparsification}}, SIAM
  Journal on Scientific Computing, 38 (2016), pp.~S332--S357.

\bibitem{Brandt:1985um}
{\sc A.~Brandt, S.~F. McCormick, and J.~Ruge}, {\em {Algebraic multigrid (amg)
  for sparse matrix equations}}, Sparsity and its Applications, 257 (1985).

\bibitem{brannick2018optimal}
{\sc J.~Brannick, F.~Cao, K.~Kahl, R.~D. Falgout, and X.~Hu}, {\em Optimal
  interpolation and compatible relaxation in classical algebraic multigrid},
  SIAM Journal on Scientific Computing, 40 (2018), pp.~A1473--A1493.

\bibitem{Brannick:2010uq}
{\sc J.~J. Brannick, A.~Frommer, K.~Kahl, S.~P. MacLachlan, and L.~T.
  Zikatanov}, {\em {Adaptive reduction-based multigrid for nearly singular and
  highly disordered physical systems}}, Electronic transactions on numerical
  analysis, 37 (2010), pp.~276--295.

\bibitem{Brezina:2010dm}
{\sc M.~Brezina, T.~A. Manteuffel, S.~F. McCormick, J.~W. Ruge, and
  G.~Sanders}, {\em {Towards Adaptive Smoothed Aggregation ($\alpha$SA) for
  Nonsymmetric Problems}}, SIAM Journal on Scientific Computing, 32 (2010),
  pp.~14--39.

\bibitem{Brezzi:2004hf}
{\sc F.~Brezzi, L.~D. Marini, and E.~S{\"u}li}, {\em {Discontinuous Galerkin
  methods for first-order hyperbolic problems}}, Mathematical models and
  methods in applied sciences, 14 (2004), pp.~1893--1903.

\bibitem{BrHeMc2000}
{\sc W.~L. Briggs, V.~E. Henson, and S.~F. McCormick}, {\em A multigrid
  tutorial}, SIAM, Philadelphia, PA, USA, 2nd~ed., 2000.

\bibitem{Brooks:1982bl}
{\sc A.~N. Brooks and T.~J. Hughes}, {\em {Streamline Upwind Petrov-Galerkin
  Formulations for Convection Dominated Flows with Particular Emphasis on the
  Incompressible Navier-Stokes Equations}}, Comput. Methods Appl. Mech. Engrg.,
  32 (1982), pp.~199--259.

\bibitem{Carvalho:2001ib}
{\sc L.~M. Carvalho, L.~Giraud, and P.~Le~Tallec}, {\em {Algebraic two-level
  preconditioners for the Schur complement method}}, SIAM Journal on Scientific
  Computing, 22 (2001), pp.~1987--2005.

\bibitem{Colomer:2013iv}
{\sc G.~Colomer, R.~Borrell, F.~X. Trias, and I.~Rodr{\'\i}guez}, {\em
  {Parallel algorithms for Sn transport sweeps on unstructured meshes}},
  Journal of Computational Physics, 232 (2013), pp.~118--135.

\bibitem{DeSterck:2008fc}
{\sc H.~De~Sterck, R.~D. Falgout, J.~W. Nolting, and U.~M. Yang}, {\em
  {Distance-two interpolation for parallel algebraic multigrid}}, Numerical
  Linear Algebra with Applications, 15 (2008), pp.~115--139.

\bibitem{pietro_ern_2011}
{\sc D.~Di~Pietro and A.~Ern}, {\em Mathematical Aspects of Discontinuous
  Galerkin Methods}, Math{\'e}matiques et Applications, Springer Berlin
  Heidelberg, 2011.

\bibitem{Dobrev:2016vc}
{\sc V.~Dobrev, T.~Kolev, and N.~A. Petersson}, {\em {Two-level convergence
  theory for parallel time integration with multigrid}}, SIAM Journal on
  Scientific Computing,  (2016).

\bibitem{bno:16}
{\sc R.~Falgout, T.~Manteuffel, B.~O'Neill, and J.~Schroder}, {\em Multigrid
  reduction in time for nonlinear parabolic problems}, Numerical Linear Algebra
  with Applications, 39 (2017), pp.~S298--S322.

\bibitem{Falgout:2014el}
{\sc R.~D. Falgout, S.~Friedhoff, T.~V. Kolev, S.~P. MacLachlan, and J.~B.
  Schroder}, {\em {Parallel Time Integration with Multigrid}}, SIAM Journal on
  Scientific Computing, 36 (2014), pp.~C635--C661.

\bibitem{Falgout:2000hs}
{\sc R.~D. Falgout and J.~E. Jones}, {\em {Multigrid on massively parallel
  architectures}}, in Multigrid methods, VI (Gent, 1999), Springer, Berlin,
  Berlin, Heidelberg, 2000, pp.~101--107.

\bibitem{Falgout:2014uz}
{\sc R.~D. Falgout and J.~B. Schroder}, {\em {Non-Galerkin coarse grids for
  algebraic multigrid}}, SIAM Journal on Scientific Computing, 36 (2014),
  pp.~C309--C334.

\bibitem{Falgout:2004cs}
{\sc R.~D. Falgout and P.~S. Vassilevski}, {\em {On Generalizing the Algebraic
  Multigrid Framework}}, SIAM Journal on Numerical Analysis, 42 (2004),
  pp.~1669--1693.

\bibitem{Falgout:2002vu}
{\sc R.~D. Falgout and U.~M. Yang}, {\em {hypre: A library of high performance
  preconditioners}}, European Conference on Parallel Processing, 2331 LNCS
  (2002), pp.~632--641.

\bibitem{Foerster:1981}
{\sc H.~Foerster, K.~St{\"u}ben, and U.~Trottenberg}, {\em Non-standard
  multigrid techniques using checkered relaxation and intermediate grids},
  Academic Press, New York,  (1981), pp.~285--300.

\bibitem{ml}
{\sc M.~Gee, C.~Siefert, J.~Hu, R.~Tuminaro, and M.~Sala}, {\em {ML} 5.0
  smoothed aggregation user's guide}, Tech. Report SAND2006-2649, Sandia
  National Laboratories, 2006.

\bibitem{sweep18}
{\sc T.~S. Haut, P.~G. Maginot, V.~Z. Tomov, B.~S. Southworth, T.~A. Brunner,
  and T.~S. Bailey}, {\em An {E}fficient {S}weep-based {S}olver for the {$S_N$}
  {E}quations on {H}igh-order {M}eshes}, Nuclear Science and Engineering,
  (2019), pp.~1--14.

\bibitem{Hawkins:2012wb}
{\sc W.~Hawkins}, {\em {Efficient massively parallel transport sweeps}}, in
  Transactions of the American Nuclear Society, Texas A and M University,
  College Station, United States, Dec. 2012, pp.~477--481.

\bibitem{Henson:2002vk}
{\sc V.~E. Henson and U.~M. Yang}, {\em {BoomerAMG: A parallel algebraic
  multigrid solver and preconditioner}}, Applied Numerical Mathematics, 41
  (2002), pp.~155--177.

\bibitem{lesaint_raviart_1974}
{\sc P.~Lesaint and P.~Raviart}, {\em On a finite element method for solving
  the neutron transport equation}, C. de Boor (Ed.), Mathematical Aspects of
  Finite Elements in Partial Differential Equations,  (1974), pp.~89--123.

\bibitem{Li:2013dm}
{\sc R.~Li and Y.~Saad}, {\em {GPU-accelerated preconditioned iterative linear
  solvers}}, The Journal of Supercomputing, 63 (2013), pp.~443--466.

\bibitem{Liu:2016ew}
{\sc W.~Liu, A.~Li, J.~Hogg, I.~S. Duff, and B.~Vinter}, {\em {A
  Synchronization-Free Algorithm for Parallel Sparse Triangular Solves}},
  European Conference on Parallel Processing, 9833 (2016), pp.~617--630.

\bibitem{Lottes:2017jz}
{\sc J.~Lottes}, {\em {Towards Robust Algebraic Multigrid Methods for
  Nonsymmetric Problems}}, Springer Theses, Springer International Publishing,
  Cham, 2017.

\bibitem{MacLachlan:2006gt}
{\sc S.~P. MacLachlan, T.~A. Manteuffel, and S.~F. McCormick}, {\em {Adaptive
  reduction-based AMG}}, Numerical Linear Algebra with Applications, 13 (2006),
  pp.~599--620.

\bibitem{Mandel:1990kv}
{\sc J.~Mandel}, {\em {On block diagonal and Schur complement
  preconditioning}}, Numerische Mathematik, 58 (1990), pp.~79--93.

\bibitem{Manteuffel:2017}
{\sc T.~A. Manteuffel, L.~N. Olson, J.~B. Schroder, and B.~S. Southworth}, {\em
  A root-node based algebraic multigrid method}, SIAM Journal on Scientific
  Computing, 39 (2017), pp.~S723--S756.

\bibitem{air2}
{\sc T.~A. Manteuffel, J.~W. Ruge, and B.~S. Southworth}, {\em Nonsymmetric
  algebraic multigrid based on local approximate ideal restriction
  ({$\ell${AIR}})}, SIAM Journal on Scientific Computing, 40 (2018),
  pp.~A4105--A4130.

\bibitem{nonsymm}
{\sc T.~A. Manteuffel and B.~S. Southworth}, {\em Convergence in norm of
  nonsymmetric algebraic multigrid}, SIAM Journal on Scientific Computing,
  (submitted).

\bibitem{Mense:2008gj}
{\sc C.~Mense and R.~Nabben}, {\em {On algebraic multi-level methods for
  non-symmetric systems {\textendash} Comparison results}}, Linear Algebra and
  its Applications, 429 (2008), pp.~2567--2588.

\bibitem{Mense:2008vx}
{\sc C.~Mense and R.~Nabben}, {\em {On algebraic multilevel methods for
  non-symmetric systems -convergence results}}, Electronic transactions on
  numerical analysis, 30 (2008), pp.~323--345.

\bibitem{Morel:2005tv}
{\sc J.~E. Morel and J.~S. Warsa}, {\em {An $S_n$ Spatial Discretization Scheme
  for Tetrahedral Meshes}}, Nuclear science and engineering, 151 (2005),
  pp.~157--166.

\bibitem{Morel:2007jj}
{\sc J.~E. Morel and J.~S. Warsa}, {\em {Spatial Finite-Element Lumping
  Techniques for the Quadrilateral Mesh $S_n$ Equations in X-Y Geometry}},
  Nuclear science and engineering, 156 (2007), pp.~325--342.

\bibitem{Notay:2000vy}
{\sc Y.~Notay}, {\em {A robust algebraic multilevel preconditioner for
  non-symmetric M-matrices}}, Numerical Linear Algebra with Applications, 7
  (2000), pp.~243--267.

\bibitem{Notay:2010em}
{\sc Y.~Notay}, {\em {Algebraic analysis of two-grid methods: The nonsymmetric
  case}}, Numerical Linear Algebra with Applications, 17 (2010), pp.~73--96.

\bibitem{notay2018}
{\sc Y.~Notay}, {\em Analysis of two-grid methods: The nonnormal case}, Tech.
  Report GANMN 18-01, 2018.

\bibitem{Olson:2011fg}
{\sc L.~N. Olson, J.~B. Schroder, and R.~S. Tuminaro}, {\em {A General
  Interpolation Strategy for Algebraic Multigrid Using Energy Minimization}},
  SIAM Journal on Scientific Computing, 33 (2011), pp.~966--991.

\bibitem{Oosterlee:1998ih}
{\sc C.~W. Oosterlee, F.~J. Gaspar, T.~Washio, and R.~Wienands}, {\em
  {Multigrid Line Smoothers for Higher Order Upwind Discretizations of
  Convection-Dominated Problems}}, Journal of Computational Physics, 139
  (1998), pp.~274--307.

\bibitem{Park:2013ju}
{\sc H.~Park, D.~A. Knoll, R.~M. Rauenzahn, C.~K. Newman, J.~D. Densmore, and
  A.~B. Wollaber}, {\em {An Efficient and Time Accurate, Moment-Based
  Scale-Bridging Algorithm for Thermal Radiative Transfer Problems}}, SIAM
  Journal on Scientific Computing, 35 (2013), pp.~S18--S41.

\bibitem{Ragusa:2012gn}
{\sc J.~C. Ragusa, J.~L. Guermond, and G.~Kanschat}, {\em {A robust
  SN-DG-approximation for radiation transport in optically thick and diffusive
  regimes}}, Journal of Computational Physics, 231 (2012), pp.~1947--1962.

\bibitem{reed_hill_1973}
{\sc W.~Reed and T.~Hill}, {\em Triangular mesh methods for the neutron
  transport equation}, Tech. Rep. LA-UR-73-479, Los Alamos Scientific
  Laboratory, Los Alamos, NM,  (1973).

\bibitem{Ries:1983}
{\sc M.~Ries, U.~Trottenberg, and G.~Winter}, {\em A note on mgr methods},
  Linear Algebra and its Applications, 49 (1983), pp.~1--26.

\bibitem{ruge:1987}
{\sc J.~W.~W. Ruge and K.~St{\"u}ben}, {\em Algebraic multigrid}, Multigrid
  methods, 3 (1987), pp.~73--130.

\bibitem{Saad:1999km}
{\sc Y.~Saad and M.~Sosonkina}, {\em {Distributed Schur complement techniques
  for general sparse linear systems}}, SIAM Journal on Scientific Computing, 21
  (1999), pp.~1337--1356.

\bibitem{Sala:2008cv}
{\sc M.~Sala and R.~S. Tuminaro}, {\em {A New Petrov{\textendash}Galerkin
  Smoothed Aggregation Preconditioner for Nonsymmetric Linear Systems}}, SIAM
  Journal on Scientific Computing, 31 (2008), pp.~143--166.

\bibitem{Schroder:2012di}
{\sc J.~B. Schroder}, {\em {Smoothed aggregation solvers for anisotropic
  diffusion}}, Numerical Linear Algebra with Applications, 19 (2012),
  pp.~296--312.

\bibitem{Treister:2015cp}
{\sc E.~Treister and I.~Yavneh}, {\em {Non-Galerkin Multigrid Based on
  Sparsified Smoothed Aggregation}}, SIAM Journal on Scientific Computing, 37
  (2015), pp.~A30--A54.

\bibitem{Vassilevski:2008wd}
{\sc P.~S. Vassilevski}, {\em {Multilevel Block Factorization
  Preconditioners}}, Matrix-based Analysis and Algorithms for Solving Finite
  Element Equations, Springer Science {\&} Business Media, Oct. 2008.

\bibitem{Wiesner:2014cy}
{\sc T.~A. Wiesner, R.~S. Tuminaro, W.~A. Wall, and M.~W. Gee}, {\em {Multigrid
  transfers for nonsymmetric systems based on Schur complements and Galerkin
  projections}}, Numerical Linear Algebra with Applications, 21 (2013),
  pp.~415--438.

\bibitem{curve}
{\sc D.~N. Woods, T.~A. Brunner, and T.~S. Palmer}, {\em High order finite
  element $s_n$ transport in $x-y$ geometry on meshes with curved surfaces}, in
  Transactions of the American Nuclear Society, vol.~114, 2016, pp.~377--380.

\bibitem{ideal18}
{\sc X.~Xu and C.-S. Zhang}, {\em On the ideal interpolation operator in
  algebraic multigrid methods},  (2018).

\bibitem{Yavneh:1998fw}
{\sc I.~Yavneh, C.~H. Venner, and A.~Brandt}, {\em {Fast multigrid solution of
  the advection problem with closed characteristics}}, SIAM Journal on
  Scientific Computing, 19 (1998), pp.~111--125.

\end{thebibliography}

\end{document}